\documentclass[12pt, english, reqno]{amsart}

\usepackage{a4wide}

\usepackage[utf8]{inputenc}
\usepackage{amsmath}
\usepackage{mathrsfs}
\usepackage{amsthm,amssymb}
\usepackage{nicefrac,amsfonts}
\usepackage{thmtools,mathtools,leftindex,leftidx}
\usepackage{caption,subcaption}
\usepackage{epsfig,graphicx,graphics,color,tikz,tikz-cd,tcolorbox}
\usetikzlibrary{tikzmark,arrows.meta}
\usepackage{enumerate,enumitem}
\usepackage[sort,nocompress,noadjust]{cite}
\usepackage{xspace}
\usepackage{bbm}
\usepackage{hyperref}
\usepackage[capitalise]{cleveref}
\usepackage{float}
\usepackage{makecell}
\usepackage{bm}
\theoremstyle{plain}

\newtheorem{theorem}{Theorem}[section]
\newtheorem{proposition}[theorem]{Proposition}
\newtheorem{corollary}[theorem]{Corollary}
\newtheorem{lemma}[theorem]{Lemma}

\theoremstyle{definition}

\newtheorem{remark}[theorem]{Remark}

\newtheorem{definition}[theorem]{Definition}
\newtheorem{example}[theorem]{Example}
\newtheorem{conjecture}[theorem]{Conjecture}

\DeclareMathOperator{\rks}{rks}
\DeclareMathOperator{\rk}{rk}
\DeclareMathOperator{\Sym}{Sym}
\DeclareMathOperator{\Gr}{Gr}
\DeclareMathOperator{\IN}{IN}
\DeclareMathOperator{\St}{St}
\DeclareMathOperator{\Tor}{Tor}

\renewcommand{\H}{\mathsf{H}}
\newcommand{\U}{\mathsf{U}}
\newcommand{\M}{\mathsf{M}}

\newcommand{\exc}{\operatorname{exc}}

\title[]{Augmented singular cohomology, uniform matroids, and real-rootedness}

\author[K. Binder]{Kyle Binder}
\address{(K.Binder)
  Department of Mathematics, Louisiana State University, Baton Rouge, United States.
}

\email{kbinde1@lsu.edu}

\author[L. Vecchi]{Lorenzo Vecchi}
\address{(L. Vecchi)
  Department of Mathematics, KTH Royal Institute of Technology, Stockholm, Sweden.
}

\email{lvecchi@kth.se}

\begin{document}
\begin{abstract}
    We study the singular cohomology rings of toric varieties associated with several fans arising from uniform matroids. 
    These rings generalize the Chow and augmented Chow rings of matroids. 
    For the singular cohomology ring arising from the augmented Bergman fan of a uniform matroid, we construct an explicit basis 
     derived from the retral basis for the singular
    cohomology ring of a uniform matroid introduced by the first author. 
    We prove that the augmented Bergman fan does not yield a singular cohomology ring that satisfies the quasi-projective Strong Lefschetz property, whereas a suitable modification of the fan does. 
    We then investigate the zeros of the corresponding refined Hodge--Poincaré polynomials. 
    For uniform matroids, we prove that the refined Hodge--Poincaré polynomials associated with both the singular cohomology ring and the modified augmented singular cohomology ring are real-rooted. The former result resolves a conjecture of the first author.
    These results extend the real-rootedness theorem of Brändén and the second author for the Chow polynomials of uniform matroids. 
    Finally, we relate the failure of real-rootedness for the augmented singular cohomology ring to the failure of Lefschetz properties.
\end{abstract}
\maketitle
\section{Introduction} 
The \emph{Chow ring} $\underline{A}^\bullet(\M)$ of a matroid $\M$ is one of the central objects in modern matroid theory. 
Introduced through the work of Feichtner and Yuzvinsky \cite{feichtnerYuzvinsky} as the Chow ring of the toric variety of a polyhedral fan $\underline{\Sigma}_{\M}$ called the \emph{Bergman fan}, it has played a fundamental role in the development of combinatorial Hodge theory and in the proof of several long-standing conjectures concerning the characteristic polynomial and the number of independent sets of a matroid \cite{AHK}. 
A different fan called the \emph{augmented Bergman fan} \(\Sigma_{\M}\) was later introduced in \cite{semismall}; the Chow ring of the associated toric variety is known as the \emph{augmented Chow ring} of \(\M\), denoted by \(A^\bullet(\M)\). 
A key feature of the Chow ring and augmented Chow ring is that they satisfy the K\"ahler package: Poincar\'e duality, the Hard Lefschetz theorem, and the Hodge--Riemann relations \cite{AHK,semismall}.

The Hilbert--Poincar\'e series of $\underline{A}^\bullet(\M)$ and \(A^\bullet(\M)\), defined as
\[
\underline{\H}_\M(x) = \sum_{j}\dim \underline{A}^j(\M) \cdot x^j \quad\text{and}\quad \H_\M(x) = \sum_{j}\dim A^j(\M) \cdot x^j,
\]
are respectively known as the \emph{Chow polynomial} and \emph{augmented Chow polynomial} of \(\M\) \cite{FMSV}. The K\"ahler package implies that both polynomials are palindromic and unimodal. 
The following conjecture has been proposed independently by Huh--Stevens \cite{Stevens} and Ferroni--Schröter \cite{ferroniSchroter}. 
\begin{conjecture}
    For every matroid \(\M\), \(\underline{\H}_\M(x)\) and \(\H_\M(x)\) have only real zeros.
\end{conjecture}
Since a polynomial with nonnegative coefficients and only real zeros has a log-concave, and hence unimodal, coefficient sequence, real-rootedness would place the known numerical properties of the (augmented) Chow ring in a stronger and more rigid framework.
The real-rootedness conjecture remains open for general matroids, but it has been established for several important families. For the class of uniform matroids, real-rootedness was settled for Chow polynomials in \cite{brandenVecchiUniform} and augmented Chow polynomials in \cite{FMSV}. Related real-rootedness results were also established for (augmented) Chow polynomials of classes of highly regular graded posets in the sense of \cite{FMV} in \cite{brandenVecchiNonnegative, CFL, hosterStump}.

In \cite{binderSingularCohomology}, the first author recently introduced the 
\emph{singular cohomology ring} of a loopless matroid $\M$. 
It is defined as the singular cohomology ring 
$ H^{\bullet}(X_{\underline{\Sigma}_\M})$
of the smooth toric variety associated with the Bergman fan $\underline{\Sigma}_\M$. Equivalently, it can be described as the Koszul homology of the Stanley--Reisner ring $\mathbb{Q}[\underline{\Sigma}_\M]$ 
viewed as an algebra over $ \Sym(N_{\mathbb{Q}}^{\vee}) $, where \(N_{\mathbb{Q}}^{\vee}\) is the rationalization of the character lattice of $ X_{\underline{\Sigma}_{\M}} $. 
This construction extends the ordinary Chow ring, as $H_{0}(K(\mathbb{Q}[\underline{\Sigma}_{\M}])) $ is naturally isomorphic to $\underline{A}^\bullet(\M)$.
Unlike the Chow ring, however, the full singular cohomology ring does not satisfy the usual K\"ahler package. 

For a uniform matroid $\U_{r,n}$ of rank $r $ on the ground set 
$[n] = \left\{1, \dots, n \right\} $, the first author constructed an explicit ``retral'' basis for the singular cohomology ring in terms of \emph{retral} and \emph{weakly retral} flags of flats together with certain admissible exterior-algebra elements \cite{binderSingularCohomologyUniform}. 
He also proved that the singular cohomology ring of a uniform matroid satisfies a version of the Strong Lefschetz property---the \emph{quasi-projective Strong Lefschetz property}---which is 
defined with respect to the associated gradeds of the weight filtration. 

Our first main contribution is to extend this theory to the augmented Bergman fan. 
We define the \emph{augmented singular cohomology ring} of a uniform matroid, that is, the singular cohomology ring of the toric variety of the augmented Bergman fan, \(H^\bullet (X_{\Sigma_{\U_{r,n}}})\).
We provide an explicit basis for this ring, using the corresponding retral bases from the non-augmented case.
For a proper flat $ F $ of the uniform matroid \(\U_{r,n}\), let
\begin{equation*}
    \mathcal B^F_{\U_{r,n}} = \left\{x_F\cdot x_{\mathcal F} \otimes \xi : x_{\mathcal F} \otimes \xi \text{ is an element of the retral basis for $H^\bullet(X_{\underline{\Sigma}_{\U_{r,n}/F}})$} \right\},
\end{equation*}
and let
\begin{equation*}
    \mathcal B^{\IN}_{\U_{r,n}} = \left\{ \prod_{j \in B}y_j \otimes \ell_J : \text{
    $B\in \mathcal B(\U_{r,n})$, $J \neq \varnothing$, $J \cap B = \varnothing$, and 
    $\min J < \min B$} \right\},
\end{equation*}
where \(\mathcal B(\U_{r,n})\) denotes the set of bases of \(\U_{r,n}\) and \(\ell_J = \ell_{j_1} \wedge\ell_{j_2} \wedge \cdots \wedge \ell_{j_{m}}\) for the ordered set \(J = \{j_1 < j_2 < \cdots < j_m\}\); the condition \(\min J < \min B\) is always satisfied if \(r=0\).
\begin{theorem}\label{thm:main-augmented-basis}
    The set \[
    \{ 1 \} \sqcup \bigsqcup_{\textrm{$F$ proper flat}}
    \mathcal{B}^{F}_{\U_{r,n}} \sqcup \mathcal{B}^{\IN}_{\U_{r,n}} \] is a basis of the augmented singular cohomology ring \(H_\bullet(K(\mathbb Q[\Sigma_{\U_{r,n}}]))\cong H^\bullet(X_{\Sigma_{\U_{r,n}}})\).
\end{theorem}
This construction generalizes the augmented Chow ring, since, in analogy with the non-augmented case, \(H_0(K(\mathbb Q[\Sigma_{\M}]))\) is naturally isomorphic to \(A^\bullet(\M)\).
As a consequence, we find a new retral basis for the augmented Chow ring of uniform matroids; see Corollary~\ref{cor:basis-augmented-chow}. 

In contrast with the non-augmented singular cohomology ring, the augmented singular cohomology ring only satisfies a partial quasi-projective Strong Lefschetz property. We identify the obstruction with the subspace generated by elements in \(\mathcal B^{\IN}_{\U_{r,n}}\).
Motivated by this, we introduce a modification of the augmented Bergman fan, denoted by \(\widetilde{\Sigma}_{\U_{r,n}}\), and study the corresponding \emph{modified augmented singular cohomology ring} of a uniform matroid, \(H^\bullet(X_{\widetilde{\Sigma}_{\U_{r,n}}})\). For this modification, the quasi-projective Strong Lefschetz property is fully recovered.

\begin{theorem}\label{thm:lefschetzModAug}
    Let $ \ell \in \Gr_{2}^{W}H^{2}(X_{\widetilde{\Sigma}_{\U_{r,n}}}) $ be generic. 
    Then the map
    \[ 
        \times \ell^{d} \colon \Gr_{q}^{W}H^{p}(X_{\widetilde{\Sigma}_{\U_{r,n}}}) \to  
                \Gr_{q + 2d}^{W}H^{p+2d}(X_{\widetilde{\Sigma}_{\U_{r,n}}})
    \]
    is injective for $ 1 \leq d \leq r + q -2p $ and surjective for $ d > r + q -2p $. 
\end{theorem}

We next study the polynomials arising from these cohomology rings. The singular cohomology ring and the corresponding Koszul homology groups carry a bigrading corresponding to the associated graded pieces of the weight filtration
\[ 
    \Gr_{2j}^{W} H^{2j-i}(X_{\underline{\Sigma}_{\M}}) \cong H_{i}\bigl(K(\mathbb{Q}[\underline{\Sigma}_{\M}])
    \bigr)_{j}.
\]
Their dimensions are encoded by the \emph{refined Hodge--Poincar\'e polynomials} 
\[ 
\underline{\H}^i_\M(x) = \sum_j \dim_{\mathbb Q} H_i\bigl(K(\mathbb Q[\underline{\Sigma}_\M])\bigr)_j \cdot x^{j-i},
\] 
where $ 0 \leq i \leq n-r $. For \(i=0\), this polynomial coincides with the Chow polynomial. 
Analogously, we define the \emph{augmented refined Hodge--Poincaré polynomials}
\[
\H^i_{\M}(x) = \sum_j \dim_{\mathbb Q} H_i\bigl(K(\mathbb Q[{\Sigma}_{\M}])\bigr)_j \cdot x^{j-i}\]
and the \emph{modified augmented refined Hodge--Poincaré polynomials} for uniform matroids 
\[\widetilde{\H}^i_{\U_{r,n}}(x) = 
\sum_{j} \dim_{\mathbb{Q}} H_{i} \bigl(K(\mathbb{Q}[\widetilde{\Sigma}_{\U_{r,n}}])\bigr)_{j} \cdot x^{j-i}.\] In degree $i=0$, we have \(H_0(K(\mathbb Q[\widetilde{\Sigma}_{\U_{r,n}}]))\cong A^\bullet(\U_{r,n})\). Consequently, both constructions recover the augmented Chow polynomial: \(\H^0_{\M}(x) = \H_{\M}(x)\) and \(\widetilde{\H}^0_{\U_{r,n}}(x) = \H_{\U_{r,n}}(x)\).

Since the singular cohomology ring does not satisfy the usual K\"ahler package, its refined Hodge--Poincar\'e polynomials need not be palindromic in general. However, the quasi-projective Strong Lefschetz property
for the singular cohomology of uniform matroids implies that the polynomials \(\underline{\H}^i_{\U_{r,n}}(x)\) are unimodal \cite{binderSingularCohomologyUniform}. This led the first author to conjecture that these polynomials are, in fact, real-rooted \cite[Conjecture~5.8]{binderSingularCohomologyUniform}.
Our second main contribution is to resolve this conjecture. 
\begin{theorem}\label{thm:main}
For every $1\leq r\leq n$ and every $0\leq i\leq n-r$, the refined Hodge--Poincar\'e polynomial $\underline{\H}^i_{\U_{r,n}}(x)$ has only real zeros. 
\end{theorem}

The theorem extends the real-rootedness result for the Chow polynomial of a uniform matroid, which is recovered for $i=0$. This provides evidence for the following general conjecture. 

\begin{conjecture}
    For every loopless matroid $\M$ of rank $r$ on the ground set $[n]$ and for 
    every $ 0 \leq i \leq n-r $, the refined Hodge--Poincar\'e polynomial 
    $\underline{\H}^i_{\M}(x)$ has only real zeros. 
\end{conjecture}

The proof of Theorem~\ref{thm:main} uses a bijective counting formula for the elements of the retral basis. It also relies on the \emph{deranged map}, a linear transformation introduced by Brändén and Solus that preserves interlacing of polynomials under suitable positivity assumptions \cite{branden-solus}.

From Theorem~\ref{thm:main-augmented-basis} we also find explicit
formulas for the augmented refined Hodge--Poincar\'e polynomials. In this setting, however, a different phenomenon occurs: the polynomials \(\H^i_{\U_{r,n}}(x)\) need not be unimodal and therefore need not be real-rooted. This failure can be viewed as a numerical shadow of the failure of the quasi-projective Strong Lefschetz property for the augmented singular cohomology ring. 
In contrast, Theorem~\ref{thm:lefschetzModAug} implies that the modified augmented refined Hodge--Poincar\'e polynomials \(\widetilde{\H}^i_{\U_{r,n}}(x)\) are unimodal. We strengthen this conclusion by proving that they are real-rooted.

\begin{theorem}\label{thm:main-augmented}
    For every $ 1 \leq r \leq n $ and $ 0 \leq i \leq n-r $, the refined Hodge--Poincar\'e polynomial 
    $\widetilde{\H}^i_{\U_{r,n}}(x)$ has only real zeros. 
\end{theorem}

Since \(\widetilde{\H}^0_{\U_{r,n}}(x) = \H_{\U_{r,n}}(x)\), Theorem~\ref{thm:main-augmented} extends the real-rootedness of the augmented Chow polynomials of uniform matroids to all modified augmented refined Hodge--Poincar\'e polynomials. Thus, the modified augmented Bergman fan extends both the Lefschetz properties and the real-rootedness phenomena associated with the augmented Chow ring. 
The proof of Theorem~\ref{thm:main-augmented} follows the same general strategy as that of Theorem~\ref{thm:main}, but uses the \emph{Eulerian transformation}. Its interlacing-preserving properties were conjectured by Brändén and Jochemko in \cite{branden-jochemko} and proved by Athanasiadis in \cite{athanasiadis}. This approach also extends the method used by Brändén and the second author to prove the real-rootedness of the Chow and augmented Chow polynomials of uniform matroids \cite{brandenVecchiUniform}.

\subsection*{Outline}
In \cref{sec:background}, we review the necessary background on fans associated with matroids, presentations of the singular cohomology rings of smooth toric varieties, and the retral basis for the singular cohomology ring associated with the Bergman fan of a uniform matroid. 
In \cref{sec:augmented-singular-cohomology} we introduce the augmented singular cohomology ring and prove Theorem~\ref{thm:main-augmented-basis}. 
In \cref{sec:modified-augmented-fan} we introduce the modified augmented Bergman fan and singular 
cohomology ring.
In \cref{sec:lefschetz-properties}, we study the Lefschetz properties of the augmented and modified augmented singular cohomology rings. More precisely, in \cref{sec:lefschetz-augmented} we describe the failure of the quasi-projective Strong Lefschetz property for the augmented singular cohomology ring, while in \cref{sec:lefschetz-modified} we prove Theorem~\ref{thm:lefschetzModAug}. 
In \cref{sec:combinatorics-weakly-retral}, we derive combinatorial formulas for the refined Hodge--Poincar\'e polynomials associated with all three rings. Finally, in \cref{sec:realRooted}, we prove Theorems~\ref{thm:main} and~\ref{thm:main-augmented}.

\subsection*{Acknowledgments}
This project was started at the 2026 Combinatorics at the Confluence conference.
We thank the organizers, Carnegie Mellon University, and the University of Pittsburgh for their hospitality. We thank Petter Brändén and Elena Hoster for helpful discussions. 
KB was partially supported by NSF grant DMS-2231492.


\section{Background}\label{sec:background}
In this section we review the necessary background on Bergman fans, the singular cohomology of smooth toric
varieties, and the retral basis for the singular cohomology of a uniform matroid defined in \cite{binderSingularCohomologyUniform}.

\subsection{The (augmented) Bergman fan}
We assume familiarity with standard matroid concepts, such as independent sets, bases, and flats. We denote by \(\mathcal B(\M)\) the set of bases of $ \M $, by \(\mathcal L(\M)\) the lattice of flats of a matroid \(\M\), and by \(\mathcal H(\M)\) the set of hyperplanes of \(\M\), i.e., the corank-1 flats. This paper focuses on uniform matroids. The uniform matroid of rank \(r\) on the ground set \([n] = \{1,2,\ldots,n\}\) is denoted by \(\U_{r,n}\). The contraction of a matroid \(\M\) by a flat \(F\) is denoted by \(\M/F\), and we will identify $ \mathcal{L}(\M/F)$ with the upper interval $ [F,\hat{1}] $ in $ \mathcal{L}(\M) $.

We recall the construction of the augmented Bergman fan $\Sigma_{\U_{r,n}}$ for uniform matroids, following \cite[Section~2.1]{semismall}. 
Let $\mathcal F$ be a flag of proper flats and $I$ an independent set of \(\U_{r,n}\). 
We say that the pair $(I,\mathcal F)$ is \emph{compatible}, and write $I \leq \mathcal F$, if $I$ is contained in every flat of $\mathcal F$. 
Notice that if $\mathcal F=\varnothing$, then $I \leq \mathcal{F}$ for any independent set $I$. 
On the other hand, if $\mathcal F \neq \varnothing$ then $I$ must not be a basis, and 
$I\subseteq \min \mathcal F$, with equality being allowed. 

Let $ N_{\U_{r,n}} = \mathbb{Z}^{n} $ be the integer lattice with basis vectors $ \left\{e_{i} \right\}_{i=1}^{n} $, and write
$ N_{\U_{r,n},\mathbb{R}} = N_{\U_{r,n}} \otimes \mathbb{R} $.
For an element $i \in [n]$ we write $\rho_i = \operatorname{cone}(e_i) \subseteq N_{\U_{r,n}, \mathbb{R}} $, and for a subset $S\subseteq [n]$ we write $\rho_S = \operatorname{cone}(-e_{[n]\setminus S}) \subseteq N_{\U_{r,n}, \mathbb{R}}$, where \(e_A = \sum_{i \in A} e_i\) for \(A\subseteq [n]\).
The next definition specializes \cite[Definition~2.4]{semismall} to the case of uniform matroids.
\begin{definition}
    For \(n > 0\), the \emph{augmented Bergman fan} \(\Sigma_{\U_{r,n}}\) of \(\U_{r,n}\) is a unimodular fan in \(N_{\U_{r,n}, \mathbb{R}}\) with rays 
    \[
    \{\rho_i \mid i\in [n]\} \cup \{\rho_F \mid |F| \leq r-1, F\subseteq [n]\}
    \]
    and cones of the form
    \[
    \sigma_{I\leq \mathcal F} = \operatorname{cone}(e_i)_{i \in I} + \operatorname{cone}(-e_{[n]\setminus F})_{F \in \mathcal F},
    \]
    where \(\mathcal F\) is a flag of proper flats of \(\mathcal L(\U_{r,n})\) and \(I\) is an independent set of \(\U_{r,n}\) compatible with \(\mathcal F\). To ease notation, we write \(\sigma_I\) for \(\sigma_{I \leq \varnothing}\). The \emph{Bergman fan} \(\underline{\Sigma}_{\U_{r,n}}\) of \(\U_{r,n}\) coincides with \(\operatorname{St}(\rho_{\varnothing})\), the star of \(\rho_{\varnothing}\) in \(\Sigma_{\U_{r,n}}\);
    see \cite[Definition (3.2.8)]{CLS} for a definition of star. 
    If \(n=0\), the augmented and non-augmented Bergman 
    fans coincide and consist of a single point.
\end{definition}

Later in \cref{sec:augmented-singular-cohomology}, we will also focus on the following subfan of \(\Sigma_{\U_{r,n}}\).
\begin{definition}
    The \emph{independence fan} \(\IN_{\U_{r,n}}\) is the unimodular subfan of \(\Sigma_{\U_{r,n}}\) consisting of all the cones of the form \(\sigma_{I}\). 
\end{definition}

\begin{example}\label{ex:U24-fan}
    Consider \(\U_{2,4}\), the uniform matroid of rank \(2\) on the ground set \(\{1,2,3,4\}\). Its augmented Bergman fan \(\Sigma_{\U_{2,4}}\) has \(9\) rays and \(14\) two-dimensional cones. In Figure \ref{fig:face-complex-ex} we draw the face complex of \(\Sigma_{\U_{2,4}}\), i.e., the complex whose vertices are rays and faces are cones. 
    \begin{figure}
    \[
    \begin{tikzpicture}[rotate around x=-30]
    \coordinate (1) at (0,0,0);
    \coordinate (2) at (0,1.5,0);
    \coordinate (3) at (0,-1,3);
    \coordinate (4) at (0,-1,-3);
    \coordinate (11) at (3,0,0);
    \coordinate (22) at (3,1.5,0);
    \coordinate (33) at (3,-1,3);
    \coordinate (44) at (3,-1,-3);
    \coordinate (e) at (6,0,0);
    \filldraw[gray]
    (1) circle[radius=2pt]
    node[font=\scriptsize, anchor=north, text=black]
    {$1 \leq \varnothing$};
    
    \filldraw[gray]
    (2) circle[radius=2pt]
    node[font=\scriptsize, anchor=south, text=black]
    {$2 \leq \varnothing$};
    
    \filldraw[gray]
    (3) circle[radius=2pt]
    node[font=\scriptsize, anchor=north, text=black]
    {$3 \leq \varnothing$};
    
    \filldraw[gray]
    (4) circle[radius=2pt]
    node[font=\scriptsize, anchor=north, text=black]
    {$4 \leq \varnothing$};
    
    \filldraw[gray]
    (11) circle[radius=2pt]
    node[font=\scriptsize, anchor=south, text=black]
    {$\varnothing \leq \{1\}$};
    
    \filldraw[gray]
    (22) circle[radius=2pt]
    node[font=\scriptsize, anchor=south, text=black]
    {$\varnothing \leq \{2\}$};
    
    \filldraw[gray]
    (33) circle[radius=2pt]
    node[font=\scriptsize, anchor=north, text=black]
    {$\varnothing \leq \{3\}$};
    
    \filldraw[gray]
    (44) circle[radius=2pt]
    node[font=\scriptsize, anchor=north, text=black]
    {$\varnothing \leq \{4\}$};
    
    \filldraw[gray]
    (e) circle[radius=2pt]
    node[font=\scriptsize, anchor=west, text=black]
    {$\varnothing \leq \{\varnothing\}$};
    \draw[very thick,dotted] 
          (1)--(2) 
          (1)--(3) 
          (1)--(4) 
          (2)--(3) 
          (2)--(4) 
          (3)--(4);
    \draw 
          (1)--(11) 
          (2)--(22) 
          (3)--(33) 
          (4)--(44);
    \draw[very thick,dashdotted] (11)--(e) 
          (22)--(e) 
          (33)--(e) 
          (44)--(e);
        \end{tikzpicture}
    \]
    \caption{The one-dimensional face complex of \(\Sigma_{\U_{2,4}}\).  The vertices correspond to the rays. The dash-dotted subcomplex corresponds to \(\underline{\Sigma}_{\U_{2,4}}\). The dotted subcomplex is the face complex of \(\IN_{\U_{2,4}}\). 
    }\label{fig:face-complex-ex}
    \end{figure}
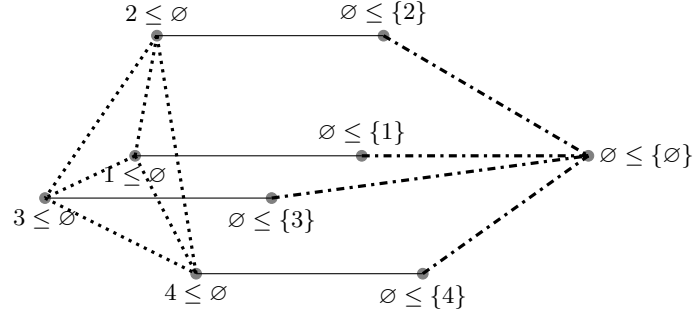
    
\end{example}

\subsection{The singular cohomology ring of a smooth toric variety}
    Given an integer lattice $ N $, any unimodular fan $ \Sigma \subseteq N_{\mathbb{R}} $ defines a smooth toric variety
    $ X_{\Sigma} $ defined over $ \mathbb{C} $. We recall here a presentation of the singular cohomology ring of 
    $ X_{\Sigma} $ in terms of the Tor algebra and Koszul homology of the Stanley--Reisner ring. We work with rational 
    coefficients throughout.

    \begin{definition}
        The \emph{Stanley--Reisner ring} of $ \Sigma $ is 
        \[ 
            \mathbb{Q}[\Sigma] = \frac{\mathbb{Q}[x_{\rho} : \textrm{$\rho \in \Sigma $ a ray}]}
            {(x_{\rho_{1}} \cdots x_{\rho_{s}} : \operatorname{cone}(\rho_{1}, \dots, \rho_{s}) \notin \Sigma).}
        \]
    \end{definition}    

    \begin{remark}
        For the augmented Bergman fan \(\Sigma_{\U_{r,n}}\) we distinguish the rays coming from elements of the ground set \(\rho_i\) from the rays associated to the rank-one flats \(\rho_{\{i\}}\) by using two different sets of variables as in \cite{semismall}. In particular, the Stanley--Reisner ring will be a quotient of 
        \[
        \mathbb Q [x_F, y_i : \text{$F$ proper flat, $i\in [n]$}].
        \]
    \end{remark}
    
    Writing $ N^{\vee} $ for the dual of $N$ and $ N^{\vee}_{\mathbb{Q}} $ for its rationalization, $ \mathbb{Q}[\Sigma] $
    is naturally a finitely-generated graded algebra over $ \Sym(N_{\mathbb{Q}}^{\vee}) $. 
    The structure map is defined by
    \[ 
            m \longmapsto \sum_{\rho \in \Sigma \textrm{ a ray}} \langle m, u_{\rho} \rangle \cdot x_{\rho}
    \]
    where $m \in N^{\vee} $ is a character and $u_{\rho} \in N $ is the primitive integral generator of $ \rho $.

    \begin{remark}
        For the integer lattice $N_{\U_{r,n}} $ and dual $ N^{\vee}_{\U_{r,n}} $, we will write
        $ \ell_{i} \in N^{\vee}_{\U_{r,n}} $ for the Kronecker dual of $ e_{i} $.
    \end{remark}

    \begin{theorem}[{\cite[Theorem 1.2]{franzRing}}]
       Let $ X_{\Sigma} $ be a smooth toric variety. As rings, 
       \begin{equation}\label{eq:torCohomology}
            H^{\bullet}(X_{\Sigma}) \cong \Tor_{\bullet}^{\Sym(N^{\vee}_{\mathbb{Q}})}(\mathbb{Q}[\Sigma], \mathbb{Q}), 
       \end{equation}
       where $ \mathbb{Q} $ is the residue field of the maximal homogeneous ideal.
    \end{theorem}

    The singular cohomology groups have an increasing \emph{weight filtration} 
    \[ 
        0 = W_{p-1}H^{p}(X_{\Sigma}) \leq W_{p}H^{p}(X_{\Sigma}) \leq \cdots \leq W_{2p} H^{p}(X_{\Sigma}) 
        = H^{p}(X_{\Sigma})
    \]
    from mixed Hodge theory. Under the isomorphism \eqref{eq:torCohomology}, the associated gradeds of the weight filtration
    correspond to the internal grading of the Tor algebra induced by the grading of $ \mathbb{Q}[\Sigma] $.

    \begin{theorem}[{\cite[Theorem~5.4]{weberWeights}}]
        Let $ X_{\Sigma} $ be a smooth toric variety. Then
        \[ 
            \Tor_{i}^{\Sym(N_{\mathbb{Q}}^{\vee})}(\mathbb{Q}[\Sigma], \mathbb{Q})_{j} \cong
            \Gr_{2j}^{W}H^{2j-i}(X_{\Sigma}) = W_{2j}H^{2j-i}(X_{\Sigma})/W_{2j-1}H^{2j-i}(X_{\Sigma}).
        \]
    \end{theorem}
    \noindent We note from this description that $ \Gr_{q}^{W}H^{p}(X_{\Sigma}) $ is trivial whenever $q$ is odd.
    
    A hands-on presentation of these Tor algebras comes from the Koszul resolution of $ \mathbb{Q} $ and Koszul homology.

    \begin{definition}
        The \emph{Koszul complex} is the bigraded chain complex
        \[ 
            K_{i}(\mathbb{Q}[\Sigma])_{j} = \mathbb{Q}[\Sigma]_{j-i} \otimes \bigwedge^{i} N^{\vee}_{\mathbb{Q}}
        \]
        with differential induced by
        \[ 
            d(x) = 0 \textrm{ for $ x \in \mathbb{Q}[\Sigma]$} \;\;\; \textrm{and} \;\;\; d(m) =  \sum_{\rho \in \Sigma \textrm{ a ray}} \langle m, u_{\rho} \rangle \cdot x_{\rho} \textrm{ for $ m \in N^{\vee} $}.
        \]
    \end{definition}    
    This complex has the structure of a graded commutative ring, with multiplication induced by
    \[ x \otimes \xi  \cdot y \otimes \zeta = xy \otimes \xi \wedge \zeta\]
    for $ x,y \in \mathbb{Q}[\Sigma] $ and $ \xi, \zeta \in \bigwedge^{\bullet} N^{\vee}_{\mathbb{Q}}. $
    The Koszul homology $ H_{\bullet}(K(\mathbb{Q}[\Sigma]))_{\bullet} $ is isomorphic to 
    $ \Tor_{\bullet}^{\Sym(N^{\vee}_{\mathbb{Q}})}(\mathbb{Q}[\Sigma], \mathbb{Q})_{\bullet} $,
    and thus $ H^{\bullet}(X_{\Sigma}) $, as a ring.

    A convenient way to organize the associated gradeds of the weight filtration is in terms of a
    \emph{Betti diagram}:
    \begin{center}
            \begin{tabular}{c|cccc}
                 & $0$ & $ 1$ & $ 2 $ & $ \cdots $  \\ \hline
                 $0$ & $ \beta_{0,0} $ & $ \beta_{1,1} $ & $ \beta_{2,2} $ & $ \cdots $ \\
                 $1$ & $ \beta_{0,1} $ & $ \beta_{1,2} $ & $ \beta_{2,3} $ & $ \cdots $ \\
                 $2$ & $ \beta_{0,2} $ & $ \beta_{1,3} $ & $ \beta_{2,4} $ & $ \cdots $ \\
                 $ \vdots $ & $ \vdots $ & $ \vdots $ & $ \vdots $ & $ \ddots $ 
            \end{tabular}
    \end{center}
    where $ \beta_{i,j} = \dim_{\mathbb{Q}} H_{i}(K(\mathbb{Q}[\Sigma]))_{j} =
    \dim_{\mathbb{Q}} \Gr_{2j}^{W} H^{2j-i}(X_{\Sigma}) $. Here we have placed the \emph{Hodge numbers} into the 
    Betti diagram, but other times we will place basis elements or the associated gradeds themselves into such a diagram.

    \subsection{Definition of (weakly) retral flags and admissible wedges}
    In this section we provide a brief overview of the \emph{retral} and 
    \emph{weakly retral} flags of flats in $ \mathcal{L}(\U_{r,n}) $ used in
    \cite{binderSingularCohomologyUniform} to construct the retral basis for the singular cohomology of 
    a uniform matroid in terms of Koszul homology. Throughout, we identify a singleton set with its unique element whenever no confusion can arise.
    Given a flag of proper flats $ \mathcal{F} $ in
    $ \mathcal{L}(\U_{r,n}) $, we define the \emph{ranks} of $ \mathcal{F} $ to be
    \[ 
        \rks(\mathcal{F}) = \left\{ \rk(F) : F \in \mathcal{F} \right\}.  
    \]
    We will follow the convention in \cite{binderSingularCohomologyUniform} that all flags of flats contain
    $ \varnothing $, or equivalently, $ 0 \in \rks(\mathcal{F}) $. For flats $ F, G \in \mathcal{L}(\M) $,
    we say that $G$ \emph{covers} $ F $ if $ F < G $ and there is no flat \(H\) such that \(F<H<G\). If \(\M\) is a uniform matroid and \(G\) is a proper flat, \(G\) covers \(F\) if and only if $ \left| G \setminus F \right| = 1 $.

    \begin{definition}
        Let $ \mathcal{F} = \left\{ \varnothing = F_{0} < \cdots < F_{s} \neq [n] \right\} $ be a flag of 
        proper flats in $ \mathcal{L}(\U_{r,n}) $. We say that $ \mathcal{F} $ is \emph{retral} if
        \[\textrm{(Cover condition): For $ 1 \leq i \leq s $,
        whenever $F_{i} $ covers $ F_{i-1} $,  
        $ F_{i} \setminus F_{i-1} =\max (F_{i+1} \setminus F_{i-1}) $.} \] Here we use the convention that 
        $ F_{s+1} = [n] $.
    \end{definition}

   In other words, retrality is a lexicographic maximality condition on the covering relations in a flag of flats.  
   For flags that include a hyperplane (a corank $ 1 $ flat), we define \emph{weakly retral} flags by
   removing this cover condition on the covering relations involving the hyperplane.

   \begin{definition}
       Let $ \mathcal{F} = \left\{ \varnothing = F_{0} < \cdots < F_{s} \neq [n] \right\} $ be a flag of
       proper flats in $ \mathcal{L}(\U_{r,n}) $. We say that $ \mathcal{F} $ is \emph{weakly retral} if
       \begin{enumerate}
           \item $ F_{s} $ is a hyperplane, and 
           \item for $ 1 \leq i < s $,
           whenever $F_{i} $ covers $ F_{i-1} $, $ F_{i} \setminus F_{i-1} =
                 \max (F_{i+1} \setminus F_{i-1}) $.
       \end{enumerate}
   \end{definition}

    For the rational vector space $ N $ with basis $ \left\{ \ell_{i} : i \in [n] \right\} $, let
    $ N_{\mathbb{Q}}^{\vee} $ be the subspace perpendicular to $ \ell_{1} + \cdots + \ell_{n} $ under the 
    canonical inner product. Thus, $ N_{\mathbb{Q}}^{\vee} $ is spanned by the elements $ \ell_{u} - \ell_{v} $
    such that $ u,v \in [n] $. We note that $ N_{\mathbb{Q}}^{\vee} $ is the rationalization of the 
    character lattice of $ X_{\underline{\Sigma}_{\U_{r,n}}} $.

    To every weakly retral flag $ \mathcal{F} $ in $ \mathcal{L}(\U_{r,n}) $, we associate a set of \emph{admissible wedges}
    in the exterior algebra $ \bigwedge^{>0} N_{\mathbb{Q}}^{\vee} $. 

    \begin{definition}\label{def:admissible}
        Let $ \mathcal{F} = \left\{ \varnothing = F_{0} < \cdots < F_{s-1} < F_{s} \neq [n] \right\} $ be a weakly
        retral flag in $ \mathcal{L}(\U_{r,n}) $. An \emph{admissible wedge} to $ \mathcal{F} $ is a basic wedge
        \[ 
            (\ell_{u_{1}} - \ell_{v_{1}}) \wedge \cdots \wedge (\ell_{u_{i}} - \ell_{v_{i}}) \in \bigwedge^{i > 0}
            N_{\mathbb{Q}}^{\vee} 
        \]
        such that
        \begin{enumerate}
            \item $ \left\{ u_{1} < \cdots < u_{i} \right\} \subseteq [n] \setminus F_{s} $,
            \item each $ v_{k} $ is the successor of $ u_{k} $ in the ordered set $ [n] \setminus F_{s} $, and
            \item if $ F_{s} $ covers $ F_{s-1} $, then $ u_{1} < (F_{s} \setminus F_{s-1}) $.
        \end{enumerate}
    \end{definition}

    These retral and weakly retral flags define the following basis for the singular cohomology ring of a uniform 
    matroid. For notation, let
    \[
        x_{\mathcal{F}} = \prod_{F \in \mathcal{F} \setminus \{\varnothing\}} x_{F} 
    \]
    for a flag of flats $ \mathcal{F} $.

    \begin{theorem}[{\cite[Theorem 3.11]{binderSingularCohomologyUniform}}]\label{thm:basis-uniform}
        The set
        \[ 
            \left\{ x_{\mathcal{F}} : \textrm{$\mathcal{F}$ is retral}\right\}
            \cup \left\{ x_{\mathcal{F}} \otimes \xi : 
            \textrm{$\mathcal{F} $ is weakly retral, $ \xi $ admissible to $ \mathcal{F} $}\right\}
        \]
        is a basis of the singular cohomology ring $H_{\bullet}(K(\mathbb{Q}[\underline{\Sigma}_{\U_{r,n}}])) \cong 
        H^{\bullet}(X_{\underline{\Sigma}_{\U_{r,n}}}) $.
    \end{theorem}

\section{Augmented singular cohomology of uniform matroids}\label{sec:augmented-singular-cohomology}
We now consider the \emph{augmented singular cohomology ring} \(H^\bullet(X_{\Sigma_{\U_{r,n}}})\)
of a uniform matroid, i.e., the singular cohomology of the toric variety associated to the augmented Bergman fan. We will use again the isomorphism with the Koszul homology of the Stanley--Reisner ring, together with the weakly retral flags from Theorem~\ref{thm:basis-uniform}, to compute a basis for it. 

\subsection{The augmented retral basis}
In this section we prove Theorem \ref{thm:main-augmented-basis}.
We do so by constructing the augmented Bergman fan step-by-step,
starting with the independence fan $ \IN_{\U_{r,n}} $ and then adding in the rays $ \rho_{F} $ in order of decreasing rank. At each step we construct a basis for the singular cohomology.
We begin with a lemma for the base case of this construction.
\begin{lemma}\label{lem:basisIndependenceFan}
    The set \( \{1\} \sqcup \mathcal{B}^{\IN}_{\U_{r,n}} \) is a basis for 
    $ H_{\bullet}(K(\mathbb{Q}[\IN_{\U_{r,n}}])) \cong H^{\bullet}(X_{\IN_{\U_{r,n}}}) $.
\end{lemma}
\begin{proof}
    We make a nested induction, first on $r $ and then on $n$. 
    
    The base case for the first induction
    is $ r = 0 $. Here, $ K(\mathbb{Q}[\IN_{\U_{0,n}}])$ is the complex $ \bigwedge^{\bullet} N_{\U_{0,n}}^{\vee} $
    with trivial differential; a basis of  $ H_{\bullet}(K(\mathbb{Q}[\IN_{\U_{0,n}}])) $ is  
    $ \{1 \} \sqcup \left\{ \ell_{J} : \textrm{$J \neq \varnothing$, $J \subseteq [n] $} \right\} $, as claimed.
    The base case for the second induction is $ \U_{n,n} $. Here, $ X_{\IN_{\U_{n,n}}}  \cong \mathbb{C}^{n} $,
    and the claimed basis is $ \{1\} $.

    For the induction, consider $ \U_{r,n} $ with $ n > r > 0 $. Recall that the \emph{closed star}  of a cone $\sigma \in \Sigma $ is the subfan
\[ \overline{\St}(\sigma) = \left\{ \tau : \operatorname{cone}(\sigma, \tau) \in \Sigma \right\}. \]
    
    We have the short exact sequence of
    $ \Sym(N_{\U_{r,n}}^{\vee}) \cong \Sym(N_{\U_{r,n-1}}^{\vee})[\ell_{n}] $-modules 
    \[ 0 \to \mathbb{Q}[\overline{\St}(\rho_{n})] \xrightarrow{\cdot y_{n}}
        \mathbb{Q}[\IN_{\U_{r,n}}] \to \mathbb{Q}[\IN_{\U_{r,n-1}}] \to 0 \]
    where $ \ell_{n} \mapsto 0 $ in $ \mathbb{Q}[\IN_{\U_{r,n-1}}] $.
    Taking Koszul homology over $ \Sym(N_{\U_{r,n}}^{\vee}) $, we get a long exact sequence
    \[ 
    \resizebox{\textwidth}{!}{
    \(\displaystyle
        \cdots \to H_{\bullet}(K(\mathbb{Q}[\IN_{\U_{r-1,n-1}}]))_{\bullet -1} \to
        H_{\bullet}(K(\mathbb{Q}[\IN_{\U_{r,n}}]))_{\bullet} \to H_{\bullet}(K(\mathbb{Q}[\IN_{\U_{r,n-1}}]))_{\bullet}
        \otimes \bigwedge^{\bullet} \mathbb{Q}\cdot \ell_{n} \to \cdots.
        \)}
    \]
    On the left, we use here the facts that $ \St(\rho_{n}) \cong \IN_{\U_{r-1,n-1}} $ and that
    $H_{\bullet}(K(\mathbb{Q}[\St(\rho_{n})])) \cong H_{\bullet}(K(\mathbb{Q}[\overline{\St}(\rho_{n})])) $
    \cite[Lemma A.3]{binderSingularCohomology}. On the right, we view $H_{\bullet}(K(\mathbb{Q}[\IN_{\U_{r,n-1}}]))$
    as Koszul homology over $ \Sym(N_{\U_{r,n-1}}^{\vee}) $ and then use the K\"{u}nneth formula.

    By the inductive hypothesis, we have bases for the terms on the left and right. From here, it is easy to
    see that the connecting homomorphism vanishes 
    except on the element $ 1 \otimes \ell_{n} $ where the
    connecting homomorphism sends it to $ 1 
    \in H_{0}(K(\mathbb{Q}[\overline{\St}(\rho_{n})])) $. 
    Therefore, we obtain a basis for 
    $ H_{\bullet}(K(\mathbb{Q}[\IN_{\U_{r,n}}]))$ by pushing forward all basis 
    elements for 
    $ H_{\bullet}(K(\mathbb{Q}[\St(\rho_{n})])) $ except $1$  and pulling back all basis elements for 
    $H_{\bullet}(K(\mathbb{Q}[\IN_{\U_{r,n-1}}])) \otimes \bigwedge^{\bullet} \mathbb{Q}\cdot \ell_{n} $ except $ 1 \otimes \ell_{n} $.
    The image under the push-forward is precisely the elements of $ \mathcal{B}^{\IN}_{\U_{r,n}} $ 
    containing $ y_{n} $, and the image under the pull-back is precisely the elements of 
    $\mathcal{B}^{\IN}_{\U_{r,n}} $ which do not contain $ y_{n} $.
\end{proof}
\begin{proof}[Proof of Theorem \ref{thm:main-augmented-basis}]
    For $ 0 \leq i \leq r $, let $ \Sigma^{i}_{\U_{r,n}} $ be the full subfan of $ \Sigma_{\U_{r,n}} $ consisting
    of the cones generated by rays $ \rho_{i} $ for $ i \in [n] $ and $ \rho_{F} $ for proper flats with $ \rk(F) \geq i $.
    In particular, $ \Sigma^{r}_{\U_{r,n}} = \IN_{\U_{r,n}} $ and $ \Sigma^{0}_{\U_{r,n}} = \Sigma_{\U_{r,n}} $.
    By a decreasing induction on $i$, we will prove that 
    \[
        \{ 1 \} \sqcup \bigsqcup_{\textrm{$\rk(F) \geq i$}}
        \mathcal{B}^{F}_{\U_{r,n}} \sqcup 
        \mathcal{B}^{\IN}_{\U_{r,n}} \] is a basis of the singular cohomology ring 
        \(H_\bullet(K(\mathbb Q[\Sigma^{i}_{\U_{r,n}}]))\cong H^\bullet(X_{\Sigma^i_{\U_{r,n}}})\).

    The base case for this induction is Lemma \ref{lem:basisIndependenceFan}. For the inductive step,
    we have the short exact sequence of $ \Sym(N_{\U_{r,n}}^{\vee}) $-modules
    \[ 
        0 \to \bigoplus_{\rk(F) = i}  \mathbb{Q}[\overline{\St}(\rho_{F})] \xrightarrow{\oplus \cdot x_{F}}
        \mathbb{Q}[\Sigma^{i}_{\U_{r,n}}] \to \mathbb{Q}[\Sigma^{i+1}_{\U_{r,n}}] \to 0.
    \]
    As the star of $ \rho_{F} $ in $ \Sigma^{i}_{\U_{r,n}} $ is $ \St(\rho_{F}) \cong \IN_{\U_{i,i}} \times \underline{\Sigma}_{\U_{r,n}/F} $, 
    and $ H_{\bullet}(K(\mathbb{Q}[\IN_{\U_{i,i}}])) = 1 $, we have the long exact sequence in
    Koszul homology
    \[
        \cdots \to \bigoplus_{\rk(F) = i}H_{\bullet}(K(\mathbb{Q}[\underline{\Sigma}_{\U_{r,n}/F}]))
        \to H_{\bullet}(K(\mathbb{Q}[\Sigma^{i}_{\U_{r,n}}])) \to H_{\bullet}(K(\mathbb{Q}[\Sigma^{i+1}_{\U_{r,n}}])) \to 
        \cdots.
    \]
    The induction gives an explicit basis for $ H_{\bullet}(K(\mathbb{Q}[\Sigma^{i+1}_{\U_{r,n}}])) $. It is a straightforward 
    check that the connecting homomorphism vanishes on these basis elements. Therefore, a basis for
    $ H_{\bullet}(K(\mathbb{Q}[\Sigma^{i}_{\U_{r,n}}])) $ is obtained by pushing forward a basis 
    for each $ H_{\bullet}(K(\mathbb{Q}[\underline{\Sigma}_{\U_{r,n}/F}])) $ and pulling back a basis for
    $ H_{\bullet}(K(\mathbb{Q}[\Sigma^{i+1}_{\U_{r,n}}])) $. Using the retral basis for the singular cohomology ring,
    the image of the push-forward is $ \bigsqcup_{\rk(F) =i} \mathcal{B}^{F}_{\U_{r,n}} $, and using the basis given by 
    induction, the image of the pull-back is $ \{1\} \sqcup \bigsqcup_{\rk(F) > i} \mathcal{B}^{F}_{\U_{r,n}} \sqcup
    \mathcal{B}^{\IN}_{\U_{r,n}} $.
\end{proof}

Since \(H_0(K(\mathbb Q[\Sigma_{\U_{r,n}}]))_\bullet\cong A^\bullet(\U_{r,n})\), Theorem \ref{thm:main-augmented-basis} consequently provides an alternative basis for the augmented Chow ring of the uniform matroid.

\begin{corollary}\label{cor:basis-augmented-chow}
    The set
    \[
    \{1\}\sqcup \{x_F \cdot x_{\mathcal F} \colon |F|\leq r-1\text{ and }\mathcal F \text{ is retral in \(\U_{r,n}/F\)} \}
    \]
    forms a basis for \(A^\bullet(\U_{r,n})\).
\end{corollary}

\begin{example}
    Consider again the uniform matroid \(\U_{2,4}\) from Example \ref{ex:U24-fan}. The basis for \(H_\bullet(K(\mathbb Q [\Sigma_{\U_{2,4}}]))\) from Theorem \ref{thm:main-augmented-basis} is written in Table \ref{tab:aug-basis-U24}. 
    In column \(0\) we find the basis for the augmented Chow ring \(A^\bullet(\U_{2,4}) \cong H_0(K(\mathbb Q[\Sigma_{\U_{2,4}}]))_\bullet\) from Corollary \ref{cor:basis-augmented-chow}.
    We highlight in bold all the basis elements that start with \(x_{\varnothing}\); these correspond, after removing the variable \(x_{\varnothing}\), to the basis elements for \(H_\bullet(K(\mathbb Q [\underline{\Sigma}_{\U_{2,4}}]))\) from Theorem \ref{thm:basis-uniform}.
    \begin{table}[h]
    \[
    \begin{tabular}{c|c|c|c}
     $\U_{2,4} $ & $0$ & $1$ & $2$   \\ \hline
     $0$ &  $1$ & $-$ & $-$  \\
     \hline
     $1$ & \makecell{$\bm{x_{\varnothing}}$ , $ x_{1} $, $ x_{2} $, \\ $ x_{3} $, $ x_{4} $} & 
     \makecell{$x_{1} \otimes(\ell_{2} -\ell_{3}) $, $ x_{1} \otimes (\ell_{3} -\ell_{4}) $, \\
     $ x_{2} \otimes (\ell_{1} - \ell_{3}) $, $ x_{2} \otimes (\ell_{3} - \ell_{4}) $, \\
     $ x_{3} \otimes (\ell_{1} - \ell_{2}) $, $ x_{3} \otimes (\ell_{2} - \ell_{4}) $, \\
     $ x_{4} \otimes (\ell_{1} - \ell_{2}) $, $ x_{4} \otimes (\ell_{2} - \ell_{3}) $} &
     \makecell{$ x_{1} \otimes (\ell_{2} - \ell_{3}) \wedge (\ell_{3} - \ell_{4}) $, \\
     $ x_{2} \otimes (\ell_{1} - \ell_{3}) \wedge (\ell_{3} - \ell_{4}) $, \\
     $ x_{3} \otimes (\ell_{1} - \ell_{2}) \wedge (\ell_{2} - \ell_{4}) $, \\
     $ x_{4} \otimes (\ell_{1} - \ell_{2}) \wedge (\ell_{2} - \ell_{3}) $}\\
     \hline
     $2$ & $\bm{x_{\varnothing} x_{4}} $ & \makecell{
     $ \bm{x_{\varnothing} x_{2} \otimes (\ell_{1} - \ell_{3})} $, \\
     $\bm{x_{\varnothing}x_{3} \otimes (\ell_{1} -\ell_{2})} $, \\
     $ \bm{x_{\varnothing} x_{3} \otimes (\ell_{2} - \ell_{4})} $, \\
     $ \bm{x_{\varnothing} x_{4} \otimes (\ell_{1} - \ell_{2}) }$,  \\
     $ \bm{x_{\varnothing} x_{4} \otimes (\ell_{2} - \ell_{3})} $,\\
     $y_{2} y_{3} \otimes \ell_{1} $, \\
     $ y_{2} y_{4} \otimes \ell_{1} $, \\ $y_{3} y_{4} \otimes \ell_{1} $, \\
     $ y_{3} y_{4} \otimes \ell_{2} $,} &
     \makecell{
     $ \bm{x_{\varnothing} x_{2} \otimes (\ell_{1} -\ell_{3}) \wedge (\ell_{3} - \ell_{4})} $, \\
     $ \bm{x_{\varnothing} x_{3} \otimes (\ell_{1} -\ell_{2}) \wedge (\ell_{2} - \ell_{4})} $, \\
     $ \bm{x_{\varnothing} x_{4} \otimes (\ell_{1} -\ell_{2}) \wedge (\ell_{2} - \ell_{3})} $, \\
     $y_{2} y_{3} \otimes (\ell_{1} \wedge \ell_{4}) $, \\
     $ y_{2} y_{4} \otimes (\ell_{1} \wedge \ell_{3}) $, \\
     $ y_{3} y_{4} \otimes (\ell_{1} \wedge \ell_{2}) $}
\end{tabular}
\]
\caption{The augmented basis from Theorem \ref{thm:main-augmented-basis} for \(\U_{2,4}\) arranged in a Betti diagram. In bold, the basis from Theorem \ref{thm:basis-uniform} for \(F = \varnothing\).}\label{tab:aug-basis-U24}
\end{table}
\end{example}

\section{Modified augmented singular cohomology of uniform matroids}\label{sec:modified-augmented-fan}
We will see in Sections \ref{sec:lefschetz-properties} and \ref{sec:realRooted} that the augmented 
singular cohomology ring does not satisfy the quasi-projective Strong Lefschetz property and that its
refined Hodge--Poincar\'{e} polynomials may not be real-rooted. In this section, we introduce a slight
modification of the augmented singular cohomology ring which recovers these properties.

\subsection{The modified augmented Bergman fan}
We now define a slight modification of the augmented Bergman fan which we use to modify the augmented
singular cohomology ring.
\begin{definition}\label{def:modified-augmented-fan} 
    Let \(1\leq r\leq n\). The \emph{modified augmented Bergman fan} \(\widetilde{\Sigma}_{\U_{r,n}}\) of \(\U_{r,n}\) is the subfan of the stellahedral fan \(\Sigma_{\U_{n,n}}\) whose cones are of the form \(\sigma_{I\leq\mathcal F}\), where every flat appearing in \(\mathcal F\) has cardinality at most \(r-1\). 
\end{definition}
We stress that, in this new definition, we have a cone \(\sigma_I\) for every subset \(I\subseteq [n]\), while in \(\Sigma_{\U_{r,n}}\) we only have such a cone for independent subsets $I$. This implies that the positive orthant \(\sigma_{[n]}\) is a cone of \(\widetilde{\Sigma}_{\U_{r,n}}\). All of its faces form the fan \(\IN_{\U_{n,n}}\). In particular, \(\widetilde{\Sigma}_{\U_{r,n}}\) is not pure when \(r < n\). 

As before, we can define the \emph{modified augmented singular cohomology ring} of a uniform matroid as the singular cohomology ring \(H^\bullet(X_{\widetilde{\Sigma}_{\U_{r,n}}})\) of the toric variety of the modified augmented Bergman fan.

\subsection{The modified augmented retral basis}
A small adaptation of Theorem \ref{thm:main-augmented-basis} provides a basis for the modified augmented singular cohomology ring
of a uniform matroid.
\begin{theorem}\label{thm:modAugBasis}
 The set \[ \{ 1 \} \sqcup \bigsqcup_{\textrm{$F$ proper flat}}
    \mathcal{B}^{F}_{\U_{r,n}} \] is a basis of the modified augmented singular cohomology ring 
    \(H_\bullet(K(
    \mathbb{Q}[\widetilde{\Sigma}_{\U_{r,n}}]))\cong H^\bullet(X_{\widetilde{\Sigma}_{\U_{r,n}}})\).
\end{theorem}

\begin{proof}
    The proof is verbatim the proof of Theorem \ref{thm:main-augmented-basis} with the following modification:
    Instead of beginning with the independence fan $ \IN_{\U_{r,n}} $, we begin with the fan $ \IN_{\U_{n,n}} $. 
    Now, $ X_{\IN_{\U_{n,n}}} \cong \mathbb{C}^{n} $, so $H^\bullet(X_{\IN_{\U_{n,n}}}) $ has basis $\{1\}$.
\end{proof}

\section{Lefschetz properties}\label{sec:lefschetz-properties}
We recall the \emph{quasi-projective Strong Lefschetz property} introduced in \cite{binderSingularCohomologyUniform}.
\begin{definition}
    Let $ X $ be a smooth, quasi-projective variety. We say that $H^{\bullet}(X) $ satisfies the
    \emph{quasi-projective Strong Lefschetz property} if for all $ p,q \in \mathbb{Z}_{\geq 0} $, $ d \geq 1$,
    and generic $ \ell \in \Gr_{2}^{W}H^{2}(X) $, the map
    \[ 
        \times \ell^{d} \colon \Gr_{q}^{W}H^{p}(X) \to \Gr_{q+2d}^{W}H^{p+2d}(X) 
    \]
    is either injective or surjective.
\end{definition}

\subsection{Partial Lefschetz for the augmented singular cohomology ring}\label{sec:lefschetz-augmented}
In this section we consider a partial Lefschetz property for the augmented singular cohomology ring. 
The main result is the following. 
\begin{theorem}\label{thm:lefschetzAug}
    Let $ \ell \in \Gr_{2}^{W}H^{2}(X_{\Sigma_{\U_{r,n}}}) $ be generic. 
    Then the map
    \[ 
        \times \ell^{d} \colon \Gr_{q}^{W}H^{p}(X_{\Sigma_{\U_{r,n}}}) \to  
                \Gr_{q + 2d}^{W}H^{p+2d}(X_{\Sigma_{\U_{r,n}}})
    \]
    is injective for $ 1 \leq d \leq r + q -2p $ and surjective for $ d > r + q -2p $,
    $ d \neq  r + q/2 -p $.
\end{theorem}

\begin{remark}\label{rem:failure-lefschetz}
    This result is best visualized in terms of a Betti diagram filled with the associated gradeds of
    cohomology (see Figure \ref{fig:bettiLefschetz}). The map $ \times \ell^{d} $ preserves columns and maps 
    row $j$ down to row $ j + d $. The result says that the maps $ \times \ell^{d} $ are injective when mapping
    from row $j$ to rows $r-j $ and above, and that they are surjective when mapping
    from row $j$ to rows $r-j+1 $ and below, except when mapping from row $j$ to row $ r $. 

    In this last case, $ \times \ell^{d} $ may fail to be either injective or surjective.
    This failure can be traced back to the basis elements in \(\mathcal B^{\IN}_{\U_{r,n}}\); see also Lemma \ref{lem:lefschetzAffine}. In the next section we will use the modified augmented Bergman fan, which is built in such a way that these specific obstacles are removed.
\end{remark}

\begin{figure}
    \centering
        {\def\arraystretch{1.5}
        \begin{tabular}{c|ccc}
             &  $ \cdots $ & $i$ & $ \cdots $  \\ \hline
             $\vdots$ & $ \ddots $ & $ \vdots $ &
             $ \ddots $ \\
             $j$  & & $ \Gr_{2i+2j}^{W}H^{i+2j}(X_{\Sigma_{\U_{r,n}}}) $ & \\
             $j+1$  & & $ \Gr_{2i+2j+2}^{W}H^{i+2j+2}(X_{\Sigma_{\U_{r,n}}}) $& \tikzmark{a} \\
             $\vdots$  &  & $ \vdots $ & \\
             $r-j$ &  & $ \Gr_{2i+2r-2j}^{W}H^{i+2r-2j}(X_{\Sigma_{\U_{r,n}}}) $ &\tikzmark{b} \\
              $r-j+1$ & & $ \Gr_{2i+2r-2j+2}^{W}H^{i+2r-2j+2}(X_{\Sigma_{\U_{r,n}}}) $ & \tikzmark{c}\\
             $\vdots$ &  & $ \vdots $ & \\
              $r-1$ &   & $ \Gr_{2i+2r-2}^{W}H^{i+2r-2}(X_{\Sigma_{\U_{r,n}}}) $ & \tikzmark{d}\\
              $r$ &   & $ \Gr_{2i+2r}^{W}H^{i+2r}(X_{\Sigma_{\U_{r,n}}}) $ &\\
             $r+1$ &  & $ 0 $ & \tikzmark{e} \\
             $\vdots$ &  & $ \vdots $ & \tikzmark{f} 
        \end{tabular}}
        \vspace*{2em}
        \begin{tikzpicture}[overlay, remember picture]
            \draw[|-|] ([yshift=1em]{pic cs:c}) to node[above, rotate =-90, midway] {surjective} ([yshift=-.25em]{pic cs:d});
            \draw[|-|] ([yshift=1em]{pic cs:a}) to node[above, rotate =-90, midway] {injective} ([yshift=-.25em]{pic cs:b});
            \draw[|->] ([yshift=.5em]{pic cs:e}) to node[above, rotate =-90, midway] {surjective} ([yshift=-2em]{pic cs:f});
        \end{tikzpicture}

    \caption{A column of the Betti diagram filled with the associated gradeds of $ H^{\bullet}(X_{\Sigma_{\U_{r,n}}}) $.
            The regions of injectivity and surjectivity for the maps $ \times \ell^{d} $ are marked on the side.
            Note the gap in surjectivity in row $ r $.}
    \label{fig:bettiLefschetz}
\end{figure}

In order to prove Theorem \ref{thm:lefschetzAug}, we first state some lemmas. 
Consider the short exact sequence of $ \Sym(N^{\vee}_{\U_{r,n}, \mathbb{Q}}) $-modules
    \[
    \resizebox{\textwidth}{!}{
    \(\displaystyle
        0 \to \displaystyle\bigoplus_{H \in \mathcal{H}(\U_{r,n})} K_{\bullet}(\mathbb{Q}[\overline{\St}(\rho_{H})]) 
        \oplus \bigoplus_{B \in \mathcal{B}(\U_{r,n})} K_{\bullet}(\mathbb{Q}[\overline{\St}(\sigma_{B})])
        \xrightarrow{\oplus \cdot x_{H} \oplus \cdot y_{B}} K_{\bullet}(\mathbb{Q}[\Sigma_{\U_{r,n}}])
        \to K_{\bullet}(\mathbb{Q}[\Sigma_{\U_{r-1,n}}]) \to 0,
    \)%
    }
    \]
    where \(y_B = \prod_{i\in B}y_i\). This yields the \emph{truncation exact sequence}
    \begin{equation}\label{eq:AugTruncation}
    \resizebox{\textwidth}{!}{
        $\displaystyle \cdots \to \displaystyle\bigoplus_{H \in \mathcal{H}(\U_{r,n})} H_{\bullet}(K(\mathbb{Q}[\overline{\St}(\rho_{H})]))
        \oplus \bigoplus_{B \in \mathcal{B}(\U_{r,n})} H_{\bullet}(K(\mathbb{Q}[\overline{\St}(\sigma_{B})]))
        \xrightarrow{\oplus \cdot x_{H} \oplus \cdot y_{B}} H_{\bullet}(K(\mathbb{Q}[\Sigma_{\U_{r,n}}]))
        \to H_{\bullet}(K(\mathbb{Q}[\Sigma_{\U_{r-1,n}}])) \to \cdots$
    }
    \end{equation} 
    after applying Koszul homology. We write the connecting homomorphism as
    \[ 
        \delta_{i} \colon H_{i}(K(\mathbb{Q}[\Sigma_{\U_{r-1,n}}])) \to
        \bigoplus_{H \in \mathcal{H}(\U_{r,n})} H_{i-1}(K(\mathbb{Q}[\overline{\St}(\rho_{H})]))
        \oplus \bigoplus_{B \in \mathcal{B}(\U_{r,n})} H_{i-1}(K(\mathbb{Q}[\overline{\St}(\sigma_{B})])) 
    \]
    and the direct sum of Gysin morphisms as 
    \[ 
        \gamma_{i} \colon \bigoplus_{H \in \mathcal{H}(\U_{r,n})} H_{i}(K(\mathbb{Q}[\overline{\St}(\rho_{H})]))
        \oplus \bigoplus_{B \in \mathcal{B}(\U_{r,n})} H_{i}(K(\mathbb{Q}[\overline{\St}(\sigma_{B})]))
        \xrightarrow{\oplus \cdot x_{H} \oplus \cdot y_{B}} H_{i}(K(\mathbb{Q}[\Sigma_{\U_{r,n}}])).
    \]
   
    \begin{lemma}\label{lem:augconnectingHom}
        For $ i \geq 1 $,  $ \gamma_{i} $ is surjective. Equivalently, 
        the connecting homomorphism $ \delta_{i} $ is injective for $ i \geq 1 $. 
    \end{lemma}

    \begin{proof}
        The equivalence of the two statements follows from the exactness of \eqref{eq:AugTruncation}. 
        To see that $ \gamma_{i} $ is surjective, we use the basis from Theorem \ref{thm:main-augmented-basis}.
        For $ i \geq 1 $, each basis element is divisible by $ x_{H} $ for some $ H \in \mathcal{H}(\U_{r,n}) $
        or by $ y_{B} $ for some $ B \in \mathcal{B}(\U_{r,n}) $. 
        The former is in the image of 
        \[
            H_{i}(K(\mathbb{Q}[\overline{\St}(\rho_{H})]))
            \xrightarrow{\cdot x_{H}} H_{i}(K(\mathbb{Q}[\Sigma_{\U_{r,n}}])),
        \]
        while the latter is in the image of 
         \[
            H_{i}(K(\mathbb{Q}[\overline{\St}(\sigma_{B})]))
            \xrightarrow{\cdot y_{B}} H_{i}(K(\mathbb{Q}[\Sigma_{\U_{r,n}}])). \qedhere
        \]
    \end{proof}
   
    \begin{lemma}\label{lem:aughardLefschetzStar}
        Let $ H \in \mathcal{H}(\U_{r,n}) $, and let $ \ell \in \Gr_{2}^{W}H^{2}(X_{\overline{\St}(\rho_{H})}) $ be 
        generic. Then the map
        \[ 
            \times \ell^{d}\colon \Gr_{q}^{W}H^{p}(X_{\overline{\St}(\rho_{H})}) \to 
            \Gr_{q+2d}^{W}H^{p+2d}(X_{\overline{\St}(\rho_{H})}) 
        \]
        is injective for $ 1 \leq d \leq r +q - 2p -1 $ and surjective for $ d \geq r + q -2p -1 $.
    \end{lemma}
    \begin{proof}
        As the isomorphism $ H^{\bullet}(X_{\St(\rho_{H})}) \cong H^{\bullet}(X_{\overline{\St}(\rho_{H})}) $
        from \cite[Lemma A.3]{binderSingularCohomology} is natural and respects the weight filtration, 
        it suffices to prove the statement for the map
        \begin{equation}\label{eq:hardLefschetzStar}
            \times \ell^{d}\colon \Gr_{q}^{W}H^{p}(X_{\St(\rho_{H})}) \to 
            \Gr_{q+2d}^{W}H^{p+2d}(X_{\St(\rho_{H})}).
        \end{equation}
        Now, $ X_{\St(\rho_{H})} \cong X_{\Sigma_{\U_{r-1,r-1}}} \times (\mathbb{C}^{*})^{n-r} $, where
        $ X_{\Sigma_{\U_{r-1,r-1}}} $ is the $(r-1)$-dimensional stellahedral toric variety. Using the K\"{u}nneth formula,
        \[ 
            \Gr_{q}^{W}H^{p}(X_{\St(\rho_{H})}) \cong H^{2p-q}(X_{\Sigma_{\U_{r-1,r-1}}}) \otimes 
                H^{q-p}((\mathbb{C}^{*})^{n-r}),
        \]
        and  the map $ \times \ell^{d} $ in \eqref{eq:hardLefschetzStar}
        is induced by the map
        \begin{equation}\label{eq:hardLefschetzStellahedral}
            \times \ell^{d} \colon  H^{2p-q}(X_{\Sigma_{\U_{r-1,r-1}}}) \to H^{2p-q+2d}(X_{\Sigma_{\U_{r-1,r-1}}}).
        \end{equation}
        As the stellahedral variety is smooth and projective, it satisfies Hard Lefschetz. Therefore,
        \eqref{eq:hardLefschetzStellahedral} is injective for $ 1 \leq d \leq r +q -2p-1 $ and 
        surjective for $ d \geq r + q - 2p- 1 $. The same then holds for \eqref{eq:hardLefschetzStar}.
    \end{proof}

    \begin{lemma}\label{lem:lefschetzAffine}
        Let \(B \in \mathcal B(\U_{r,n})\).
        For $ d \geq 1 $, the zero map
        \[ 
            0^{d} \colon \Gr_{q}^{W} H^{p}(X_{\overline{\St}(\sigma_{B})}) \to
            \Gr_{q+2d}^{W} H^{p+2d}(X_{\overline{\St}(\sigma_{B})})
        \] 
        is injective for $ q \neq 2p $ and surjective for $ d \neq q/2 - p $.
    \end{lemma}
     \begin{proof}
        It suffices to compute the cohomology of \(X_{\St(\sigma_B)}\), which is isomorphic to \(\mathbb C^r \times (\mathbb C^*)^{n-r}\). Using the Künneth formula, we write
        \[ 
        H^p(X_{\St(\sigma_B)}) \cong H^p((\mathbb C^*)^{n-r}). 
        \]
        The mixed Hodge structure is pure of weight \(2p\), hence
        \[
        \Gr_q^W H^p ((\mathbb C^*)^{n-r}) = \begin{cases}
            0 &q\neq 2p\\ \bigwedge^p \mathbb Q^{n-r} & q=2p.
        \end{cases}
        \]

        Therefore, if $ q \neq 2p $, the domain of $0^{d}$ is trivial. Similarly, the codomain of
        $ 0^{d} $ is trivial for $ d \neq q/2 - p $.
    \end{proof}

    \begin{proof}[Proof of Theorem \ref{thm:lefschetzAug}]
       If $ q = p $ or $ r = n $, then the statement holds from the Hard Lefschetz property of the augmented Chow ring
       or the stellahedral variety, respectively. Therefore, assume that $ q > p $ and $ r < n $.
       To prove injectivity, lift $\ell$ to an element $ \hat{\ell} \in \Gr_{2}^{W}H^{2}(X_{\Sigma_{\U_{r+1,n}}}) $,
       and for each $ H \in \mathcal{H}(\U_{r+1,n}) $, restrict $ \hat{\ell} $ to an element
       $ \hat{\ell}_{H} \in \Gr_{2}^{W}H^{2}(X_{\overline{\St}(\rho_{H})}) $. By genericity, we may assume each $ \hat{\ell}_{H} $ 
       is also generic. Consider the commutative square induced by applying $ \ell^{d} $ to the truncation exact sequence
       \begin{center} \resizebox{\linewidth}{!}{%
       \begin{tikzcd}[cramped, column sep=small, ampersand replacement=\&] 
       \Gr_{q}^{W}H^{p}(X_{\Sigma_{\U_{r,n}}}) \dar["\times \ell^{d}"] \rar["\delta_{q-p}"] 
       \& \displaystyle \bigoplus_{H \in \mathcal{H}(\U_{r+1,n})} \Gr_{q-2}^{W} H^{p-1}(X_{\overline{\St}(\rho_{H})}) \oplus \bigoplus_{B \in \mathcal{B}(\U_{r+1,n})} \Gr_{q-2r-2}^{W} H^{p-2r-1}(X_{\overline{\St}(\sigma_{B})}) \dar["\times\hat{\ell}^{d}_{H} \oplus \times 0^{d}"] \\ 
       \Gr_{q+2d}^{W} H^{p+2d}(X_{\Sigma_{\U_{r,n}}}) \rar["\delta_{q-p}"] 
       \& \displaystyle \bigoplus_{H \in \mathcal{H}(\U_{r+1,n})} \Gr_{q+2d-2}^{W} H^{p+2d-1}(X_{\overline{\St}(\rho_{H})}) \oplus \bigoplus_{B \in \mathcal{B}(\U_{r+1,n})} \Gr_{q-2r+2d-2}^{W} H^{p-2r+2d-1}(X_{\overline{\St}(\sigma_{B})}). 
       \end{tikzcd}%
       } 
       \end{center}
       By Lemma \ref{lem:augconnectingHom}, the horizontal arrows are injective. 
       When $ 1 \leq d \leq r + q -2p $, we have $ q-2r -2 > 2p -4r - 2 $, and by
       Lemmas \ref{lem:aughardLefschetzStar} and \ref{lem:lefschetzAffine}, the vertical arrow on the right is injective.
       This forces the vertical arrow on the left to be injective for
       $ 1 \leq d \leq r + q -2p $ as well.

       To prove surjectivity, we use the commutative square
       \begin{center}
            \resizebox{\linewidth}{!}{%
       \begin{tikzcd}[cramped, column sep=small, ampersand replacement=\&] 
                \displaystyle\bigoplus_{H \in \mathcal{H}(\U_{r,n})}\Gr_{q-2}^{W}H^{p-2}(X_{\overline{\St}(\rho_{H})}) 
                \oplus
                \bigoplus_{B \in \mathcal{B}(\U_{r,n})} \Gr_{q-2r}^{W} H^{p-2r}(X_{\overline{\St}(\sigma_{B})})
                \rar["\gamma_{q-p}"] \dar["\times\ell^{d}_{H} \oplus 0^{d}"] \&
                \Gr_{q}^{W}H^{p}(X_{\Sigma_{\U_{r,n}}}) \dar["\times \ell^{d}"] \\
                \displaystyle\bigoplus_{H \in \mathcal{H}(\U_{r,n})}\Gr_{q+2d-2}^{W}H^{p+2d-2}(X_{\overline{\St}(\rho_{H})}) 
                \oplus
                \bigoplus_{B \in \mathcal{B}(\U_{r,n})} \Gr_{q+2d-2r}^{W} H^{p+2d-2r}(X_{\overline{\St}(\sigma_{B})})
                \rar["\gamma_{q-p}"] \&
                \Gr_{q+2d}^{W}H^{p+2d}(X_{\Sigma_{\U_{r,n}}}).
            \end{tikzcd} }      
       \end{center}
       The horizontal arrows are surjective by Lemma \ref{lem:augconnectingHom}. 
       By Lemmas \ref{lem:aughardLefschetzStar} and \ref{lem:lefschetzAffine}, the vertical arrow on the left is surjective for
       $ d \geq r+q-2p+1 $ whenever $ d \neq r + q/2 - p $. This forces the vertical arrow on the right to be surjective for 
       $ d \geq r+q-2p+1 $, $ d \neq r + q/2 - p $ as well.
    \end{proof}

\subsection{Lefschetz for the modified augmented singular cohomology ring}\label{sec:lefschetz-modified}
We now show that the modified augmented singular cohomology ring recovers the full quasi-projective
Strong Lefschetz property. Our proof here is largely the same as the one in the previous section.

    Consider the short exact sequence of $ \Sym(N^{\vee}_{\U_{r,n}, \mathbb{Q}}) $-modules
    \[ 
        0 \to \bigoplus_{H \in \mathcal{H}(\U_{r,n})} K_{\bullet}(\mathbb{Q}[\overline{\St}(\rho_{H})])
        \xrightarrow{\oplus \cdot x_{H}} K_{\bullet}(\mathbb{Q}[\widetilde{\Sigma}_{\U_{r,n}}])
        \to K_{\bullet}(\mathbb{Q}[\widetilde{\Sigma}_{\U_{r-1,n}}]) \to 0.
    \]
    Applying Koszul homology yields the \emph{truncation exact sequence}
    \begin{equation}\label{eq:modAugTruncation}
        \cdots \to \bigoplus_{H \in \mathcal{H}(\U_{r,n})} H_{\bullet}(K(\mathbb{Q}[\overline{\St}(\rho_{H})]))
        \xrightarrow{\oplus \cdot x_{H}} H_{\bullet}(K(\mathbb{Q}[\widetilde{\Sigma}_{\U_{r,n}}]))
        \to H_{\bullet}(K(\mathbb{Q}[\widetilde{\Sigma}_{\U_{r-1,n}}])) \to \cdots.
    \end{equation} 
    We write the connecting homomorphism as
    \[ 
        \delta_{i} \colon H_{i}(K(\mathbb{Q}[\widetilde{\Sigma}_{\U_{r-1,n}}])) \to
        \bigoplus_{H \in \mathcal{H}(\U_{r,n})} H_{i-1}(K(\mathbb{Q}[\overline{\St}(\rho_{H})]))
    \]
    and the direct sum of Gysin morphisms as 
    \[ 
        \gamma_{i} \colon \bigoplus_{H \in \mathcal{H}(\U_{r,n})} H_{i}(K(\mathbb{Q}[\overline{\St}(\rho_{H})]))
        \xrightarrow{\oplus \cdot x_{H}} H_{i}(K(\mathbb{Q}[\widetilde{\Sigma}_{\U_{r,n}}])).
    \]

    \begin{lemma}\label{lem:connectingHom}
        For $ i \geq 1 $,  $ \gamma_{i} $ is surjective. Equivalently, 
        the connecting homomorphism $ \delta_{i} $ is injective for $ i \geq 1 $. 
    \end{lemma}

    \begin{proof}
        The equivalence of the two statements follows from exactness of the truncation exact sequence. 
        To see that $ \gamma_{i} $ is surjective, we use the basis from Theorem \ref{thm:modAugBasis}.
        For $ i \geq 1 $, each basis element is divisible by $ x_{H} $ for some $ H \in \mathcal{H}(\U_{r,n}) $
        and therefore in the image of 
        \[
            H_{i}(K(\mathbb{Q}[\overline{\St}(\rho_{H})]))
            \xrightarrow{\cdot x_{H}} H_{i}(K(\mathbb{Q}[\widetilde{\Sigma}_{\U_{r,n}}])).\qedhere
        \]
    \end{proof}
   
    \begin{proof}[Proof of Theorem \ref{thm:lefschetzModAug}]
       If $ q = p $ or $ r = n $, then the statement holds from the Hard Lefschetz property of the augmented Chow ring
       or the stellahedral variety, respectively. Therefore, assume that $ q > p $ and $ r < n $.

       To prove injectivity, lift $\ell$ to an element $ \hat{\ell} \in \Gr_{2}^{W}H^{2}(X_{\widetilde{\Sigma}_{\U_{r+1,n}}}) $,
       and for each $ H \in \mathcal{H}(\U_{r+1,n}) $, restrict $ \hat{\ell} $ to an element
       $ \hat{\ell}_{H} \in \Gr_{2}^{W}H^{2}(X_{\overline{\St}(\rho_{H})}) $. By genericity, we may assume each $ \hat{\ell}_{H} $
       is also generic. Consider the commutative square induced by applying $ \times \ell^{d} $ to the truncation exact sequence
       \begin{center}
           \begin{tikzcd}
                \Gr_{q}^{W}H^{p}(X_{\widetilde{\Sigma}_{\U_{r,n}}}) 
                \dar["\times \ell^{d}"] \rar["\delta_{q-p}"]
                & \displaystyle\bigoplus_{H \in \mathcal{H}(\U_{r+1,n})} \Gr_{q-2}^{W} H^{p-1}(X_{\overline{\St}(\rho_{H})}) \dar["\times
                \hat{\ell}^{d}_{H}"] \\
                 \Gr_{q +2d}^{W}H^{p+2d}(X_{\widetilde{\Sigma}_{\U_{r,n}}}) \rar["\delta_{q-p}"]
                & \displaystyle\bigoplus_{H \in \mathcal{H}(\U_{r+1,n})} \Gr_{q+2d-2}^{W} H^{p+2d-1}(X_{\overline{\St}(\rho_{H})}).
           \end{tikzcd}
       \end{center}
       By Lemma \ref{lem:connectingHom}, the horizontal arrows are injective. By
       Lemma \ref{lem:aughardLefschetzStar}, the vertical arrow on the right is injective for 
       $ 1 \leq d \leq r + q - 2p  $, as the closed star of $ \rho_{H} $ in
       $ \Sigma_{\U_{r+1,n}} $ is isomorphic to the closed star of $ \rho_{H} $ in $ \widetilde{\Sigma}_{\U_{r+1,n}} $. 
       This forces the vertical arrow on the left to be injective for
       $ 1 \leq d \leq r + q -2p $ as well.

       To prove surjectivity, we use the commutative square
       \begin{center}
            \begin{tikzcd}
                \displaystyle\bigoplus_{H \in \mathcal{H}(\U_{r,n})}\Gr_{q-2}^{W}H^{p-2}(X_{\overline{\St}(\rho_{H})}) 
                \rar["\gamma_{q-p}"] \dar["\times \ell_{H}^{d}"] &
                \Gr_{q}^{W}H^{p}(X_{\widetilde{\Sigma}_{\U_{r,n}}}) \dar["\times \ell^{d}"] \\
                \displaystyle\bigoplus_{H \in \mathcal{H}(\U_{r,n})}\Gr_{q+2d-2}^{W}H^{p+2d-2}(X_{\overline{\St}(\rho_{H})}) 
                \rar["\gamma_{q-p}"] &
                \Gr_{q+2d}^{W}H^{p+2d}(X_{\widetilde{\Sigma}_{\U_{r,n}}}).
            \end{tikzcd}       
       \end{center}
       The horizontal arrows are surjective by Lemma \ref{lem:connectingHom}. 
       By Lemma \ref{lem:aughardLefschetzStar}, the vertical arrow on the left is surjective for
       $ d > r+q-2p-1 $. This forces the vertical arrow on the right to be surjective for 
       $ d > r+q-2p-1 $ as well.
    \end{proof}

\section{Combinatorics of weakly retral flags}\label{sec:combinatorics-weakly-retral}
We now study the polynomials arising from these three cohomology rings. The singular cohomology ring and the corresponding Koszul homology groups carry a bigrading corresponding to the associated graded pieces of the weight filtration
\[
\Gr_{2j}^W H^{2j-i} (X_{\underline{\Sigma}_{\U_{r,n}}}) \cong H_i (K(\mathbb Q [\underline{\Sigma}_{\U_{r,n}}]))_j.
\]
Their dimensions are encoded by the \emph{refined Hodge--Poincaré polynomials},
\[
\underline{\H}^i_{\U_{r,n}}(x) = \sum_{j}\dim_{\mathbb Q} H_i (K(\mathbb Q [\underline{\Sigma}_{\U_{r,n}}]))_j \cdot x^{j-i},
\]
where \(0 \leq i \leq n-r\). Analogously, we consider the \emph{(modified) augmented refined Hodge--Poincaré polynomials},
\[
{\H}^i_{\U_{r,n}}(x) = \sum_{j}\dim_{\mathbb Q} H_i (K(\mathbb Q [{\Sigma}_{\U_{r,n}}]))_j \cdot x^{j-i} \quad \text{and} \quad \widetilde{\H}^i_{\U_{r,n}}(x) = \sum_{j}\dim_{\mathbb Q} H_i (K(\mathbb Q [\widetilde{\Sigma}_{\U_{r,n}}]))_j \cdot x^{j-i}.
\]
Formulas for \(\underline{\H}^i_{\U_{r,n}}(x)\) were presented and discussed in \cite{binderSingularCohomologyUniform}.
We now describe a new construction for pairs of weakly retral flags and admissible wedges, which leads to a new formula for the refined Hodge--Poincar\'{e} polynomials (Lemma~\ref{lem:hpDerangement}). This will be used in \cref{sec:counts-augmented-basis} to provide a formula for the augmented refined Hodge--Poincar\'{e} polynomials
and in \cref{sec:realRooted} to deduce our real-rootedness results. 

Let $\mathfrak S_{n}$ denote the symmetric group on $n$ letters, and let $\mathfrak D_{n}\subseteq\mathfrak S_{n}$ denote the set of derangements. 
For $\omega\in\mathfrak S_{n}$, let $\exc(\omega) = \bigl|\{k\mid\omega(k)>k\}\bigr|$ be the number of excedances of $\omega$.
The Eulerian and derangement polynomials are 
\[ 
A_n(x) = \sum_{\omega \in \mathfrak S_{n}}x^{\exc(\omega)} \quad
\textrm{and} \quad
d_{n}(x) = \sum_{\omega\in\mathfrak D_{n}}x^{\exc(\omega)}. 
\]
By convention, we define $ d_{0}(x) = 1 $. The Eulerian and derangement polynomials are related by the following identity
\begin{equation}\label{eq:eulerian-derangement-identity}
    A_n(x) = \sum_{k=0}^{n}\binom{n}{k}d_k(x).
\end{equation}

    \subsection{Counts of weakly retral flags and admissible wedges}
        We now give a method for counting and constructing the following class of rank-selected 
        weakly retral flags and admissible wedges.

        \begin{definition}
            For $ 0 \leq j \leq r-1 $ and $ 1 \leq i \leq n-r $, let 
            \[ 
                W_{\{ 0, j, j+1, \dots, r-1 \}, i}(\U_{r,n}) = 
                \left\{ (\mathcal{F}, \xi) : \; \begin{aligned}
                    &\textrm{$\mathcal{F}$ is weakly retral,
                    $ \xi \in \bigwedge^{i} N_{\mathbb{Q}}^{\vee} $ admissible to 
                    $ \mathcal{F} $,} \\ 
                    &\textrm{and $ \rks(\mathcal{F}) = \left\{ 0, j, j+1, \dots, r-1 \right\} $} \\
                \end{aligned} 
                \right\}.
            \]
            For ease of notation, we will write this as $ W_{j,i}(\U_{r,n}) $.
            Note that $ W_{0,i}(\U_{r,n}) = W_{1,i}(\U_{r,n}) $.
        \end{definition}  

        Along with the derangement polynomials $ d_{j}(x) $, the cardinalities of $ W_{j,i}(\U_{r,n}) $ 
        provide formulas for the refined Hodge--Poincar\'{e} polynomials.

        \begin{lemma}\label{lem:hpDerangement}
            For $ i \geq 1 $,
            \[ 
                \underline{\H}^{i}_{\U_{r,n}}(x) = \sum_{j=0}^{r-1} \left| W_{j,i}(\U_{r,n}) \right| \, d_{j}(x) \cdot x^{r-j-1}.
            \]
        \end{lemma}   
        \begin{proof}
            This follows from \cite[Proposition 4.8 and Theorem 4.15]{binderSingularCohomologyUniform},
            which say that 
            \[ 
                \left| W_{j,i}(\U_{r,n}) \right| = 
                \sum_{\ell=1}^{n-r-i+1} \, \binom{n}{j} \binom{n-j-\ell}{r-j-1} \binom{n-r-\ell}{i-1},
            \]
            and
            \[ 
                \underline{\H}_{\U_{r,n}}^{i}(x) = \sum_{j=0}^{r-1}\sum_{\ell=1}^{n-r-i+1} 
                \, \binom{n}{j} \binom{n-j-\ell}{r-j-1} \binom{n-r-\ell}{i-1}
                \, d_{j}(x) \cdot x^{r-j-1}. \qedhere
            \]
        \end{proof} 
        \noindent We note that, as $ d_{1}(x)=  0 $, $ \left| W_{1,i}(\U_{r,n}) \right| $ plays no role in the formula.
        
        We now describe a method of counting and constructing
        elements of $ W_{j,i}(\U_{r,n}) $. We begin by constructing weakly retral flags $ \mathcal{F} $ with
        $ \rks(\mathcal{F}) = \left\{0, j, \dots, r-1\right\} $.
        
        \begin{lemma}\label{lem:decreasingWeaklyRetral}
            Fix a uniform matroid $ \U_{r,n} $, and let $ 0 \leq j \leq r-1 $ with
            $j \neq 1 $.
            Given a hyperplane $H$ and a flat $ G \leq H $ of rank $j$, there is a unique weakly retral flag
            $ \mathcal{F} $ which contains $G$ and $H$ and has $ \rks(\mathcal{F}) = \left\{0, j, \dots, r-1 \right\} $.

            Moreover, if $H$ covers the flat $F$ in $ \mathcal{F} $, then $ H \setminus F = \min (H \setminus G) $.
        \end{lemma}

        \begin{proof}
            If \(r=1\) there is nothing to prove. 
            When \(r\geq 2\), write $ H \setminus G = \left\{ x_{1} > \dots > x_{r-j-1} \right\} $ and consider
            \[ \mathcal{F} = \left\{ \varnothing < G < G \cup \{x_{1}\} < G \cup \{ x_{1}, x_{2} \} <  \cdots < 
            G \cup \{ x_{1}, \dots, x_{r-j-2}\} < H \right\}. \]
            If $j=0$, then \(G=\varnothing\), which we do not list twice in the flag. Similarly, for $j=r-1$ and \(G = H\). So we can consider \(1<j<r-1\).
            We claim that \(\mathcal{F}\) is the unique weakly retral chain which contains $G$ and $H$ and has 
            $ \rks(\mathcal{F}) = \left\{0, j, \dots, r-1 \right\} $. Indeed, the cover condition 
            for $ G \cup \{x_{1}, \dots, x_{k-1} \} < G \cup \{ x_{1}, \dots, x_{k}\} $ is equivalent
            to the condition that $ x_{k} > x_{k+1} $.
            Therefore, a flag of flats 
            \[ 
                \mathcal{F}' = \left\{ \varnothing < G < G \cup \{x'_{1}\} <  \cdots < 
            G \cup \{ x'_{1}, \dots, x'_{r-j-2}\} < G \cup \{ x'_{1}, \dots, x'_{r-j-1}\} = H \right\} 
            \]
            with $ \rks(\mathcal{F}') = \left\{0,j,\dots,r-1\right\} $ is weakly retral if and only if
            $ H \setminus G = \left\{ x'_{1} > \cdots > x'_{r-j-1} \right\} $. (Note that the assumption 
            $ j > 1 $ implies $ G $ does not cover $\varnothing$, so there is no restriction on $ G $. Similarly,
            weakly retral flags do not impose any conditions on the hyperplane $H$.) 
            
            The last statement follows immediately.
        \end{proof}
        We now extend this to a construction of pairs $ (\mathcal{F}, \xi) \in W_{j,i}(\U_{r,n}) $.
        Given an ordered set $ A = \left\{a_{1} < \dots < a_{n} \right\} $, 
        we let $ \min_{i} A = \left\{ a_{i} \right\} $ and $ \min_{i}^{j} A = \left\{ a_{i} < \dots < a_{j} \right\} $.
        For a set $S$ and integer $n$, we let $ \binom{S}{n}$ be the set of subsets of $ S $ with cardinality $n$.

\begin{lemma}\label{lem:count-retral}
    Let $ 0 \leq j \leq r-1 $, $j\neq 1$, and $ 1 \leq i \leq n-r $. Elements of $ W_{j,i}(\U_{r,n}) $ 
    are in bijection with the following data:
    \begin{enumerate}
        \item an integer $ 0 \leq k \leq j $,
        \item an integer $ 0 \leq m \leq r-k-1 $,
        \item a subset $ A \in \binom{[n-m-1]}{r-m+i-1} $,
        \item a subset $ B \in \binom{A \setminus \min_{1}^{k+1} A}{i-1} $, and 
        \item a subset $ C \in \binom{(A \cup \{n-m+1, \dots, n\}) \setminus (B \cup \min_{1}^{k+1} A)}{j-k} $.
    \end{enumerate}
    Consequently,
    \[ 
        \left| W_{j,i}(\U_{r,n}) \right| = \sum_{k=0}^{j} \sum_{m=0}^{r-k-1} \binom{n-m-1}{r-m+i-1}
        \binom{r-m+i-1-(k+1)}{i-1} \binom{r-k-1}{j-k}.
    \]
\end{lemma}

\begin{proof}
   Let $ D $ denote the set of data $ (k,m,A,B,C) $ satisfying conditions (1)--(5). We construct a map
   $ \phi\colon W_{j,i}(\U_{r,n}) \to D $ which we show is a bijection.

   \vspace*{1em}
   \noindent\textbf{Definition of $ \phi $:}
   Fix $ (\mathcal{F}, \xi) \in W_{j,i}(\U_{r,n}) $. Let $G$ be the rank $ j$ flat of 
   $ \mathcal{F} $, let $ H $ be the hyperplane of $ \mathcal{F} $, and let $ \xi = (\ell_{u_{1}} -
   \ell_{v_{1}}) \wedge \cdots \wedge (\ell_{u_{i}} - \ell_{v_{i}}) $ as in 
   Definition \ref{def:admissible}. We then define $ \phi(\mathcal{F}, \xi) =
   (k,m,A,B,C) $, where 
   \begin{enumerate}
       \item $k$ is the largest integer such that $ G $ contains 
            $\min_{1}^{k} (H \cup \left\{u_{1}, \dots, u_{i} \right\}) $,
       \item $m$ is the largest integer such that $H$ contains $ \left\{n-m+1, \dots, n \right\} $,
       \item $A = \left( H \cup \left\{u_{1}, \dots, u_{i} \right\} \right) \setminus \left\{n-m+1, \dots, n\right\}$,
       \item $B = \left\{ u_{2}, \dots, u_{i} \right\}$, and
       \item $C = G \setminus \min_{1}^{k} G $.
   \end{enumerate}
   We now verify that $ \phi(\mathcal{F}, \xi) \in D $. 
   \begin{enumerate}
       \item As $ \left| G \right| = j $, we find that $ 0\leq k \leq j $.
        \item Note that $ \min_{1}^{k}(H \cup \left\{ u_{1}, \dots, u_{i} \right\}) \cap
              \left\{n-m+1, \dots, n \right\} = \varnothing $, because 
              \[ \textstyle\min_{k} (H \cup \left\{u_{1}, \dots, u_{i} \right\}) < u_{1} < n-m+1. \] 
                As $ \min_{1}^{k}(H \cup \left\{ u_{1}, \dots, u_{i} \right\}) \cup
                \left\{ n-m+1, \dots,n \right\} \subseteq H $, we then find that $ k + m \leq r-1 $ 
                and $ 0 \leq m \leq r- k-1 $.
        \item It is clear that $ \left| A \right| = r-m+i-1 $. We must show that $ A \subseteq [n-m-1] $, and 
        for this it suffices to show that $n-m \notin (H \cup \left\{ u_{1}, \dots, u_{i} \right\}) $.
        By the maximality of $m$, $ n-m \notin H $. As $ u_{1} < \cdots < u_{i} < v_{i} < n-m+1 $, we see that 
        $ n-m \notin \left\{u_{1}, \dots, u_{i} \right\} $ as well.
        \item Clearly $ \left| B \right| = i-1 $. As $B$ is disjoint from $ H $,
        we find that $ B \subseteq A \setminus \min_{1}^{k} A $. To show that $ \min_{k+1} A \notin B $, we will
        show that $ u_{1} = \min_{k+1} A $. By the maximality of $k$, either $ \min_{k+1} A \in H \setminus G $ 
        or $ \min_{k+1} A = u_{1} $. If $ \min_{k+1} A \in H \setminus G $, then $ \min_{k+1} A = \min (H \setminus G) $.
        By Lemma \ref{lem:decreasingWeaklyRetral}, $ u_{1} < \min_{k+1} A $ which is impossible, as
        $ u_{1} \in A $ and $ \min_{k} A < u_{1} $ as well. Therefore $u_{1} = \min_{k+1} A $.
        \item It is clear that $ \left| C \right| = j-k $. It remains to see that $ C \subseteq 
        (H \cup \left\{u_{1}, \dots, u_{i} \right\}) \setminus (B \cup \min_{1}^{k+1} A) $. This follows
        as $ G \subseteq H $, and $ B \cup \min_{1}^{k+1} A = \min_{1}^{k} G \cup \left\{u_{1}, \dots, u_{i} \right\} $.
    \end{enumerate}
        \vspace*{1em}
        \noindent \textbf{Injectivity of $\phi$:}
        To see that $\phi$ is an injection, we show that we can recover $ G $ and $H$---the rank $j$ and rank $r-1$
        flats of $ \mathcal{F} $, respectively---and $ \left\{u_{1}, \dots, u_{i} \right\} $
        from $ \phi(\mathcal{F}, \xi) = (k,m,A,B,C) $. By
        Lemma \ref{lem:decreasingWeaklyRetral}, $ G$ and $H$ determine $ \mathcal{F} $, while $ \left\{u_{1},\dots,
        u_{i} \right\} $ determines $ \xi $ by the definition of admissible wedges. Above we showed that
        $ u_{1} = \min_{k+1} A $, so $ \left\{ u_{1}, \dots, u_{i} \right\} = B \cup \min_{k+1} A $. 
        Then $H = A \cup \left\{n-m+1,\dots, n\right\} \setminus (B \cup \min_{k+1} A) $, and
        $ G = C \cup \min_{1}^{k} A $.

        \vspace*{1em}
        \noindent \textbf{Surjectivity of $\phi$:} 
        Finally, we show that $\phi$ is surjective. Given $(k,m,A,B,C) \in D $, let
        $\left\{u_{1} < \dots < u_{i} \right\} = B \cup \min_{k+1} A $, $ H = (A \cup \left\{n-m+1,\dots,n\right\})
        \setminus \left\{u_{1}, \dots, u_{i} \right\} $, 
        and $ G = C \cup \min_{1}^{k} A $. By Lemma \ref{lem:decreasingWeaklyRetral}, there is a unique weakly retral
        flag $ \mathcal{F} $ containing $G $ and $H$ with $ \rks(\mathcal{F}) = \left\{0, j, \dots, r-1 \right\} $.
        Define $ \xi $ to be the basic wedge
        \[ 
            (\ell_{u_{1}} - \ell_{v_{1}}) \wedge \cdots \wedge (\ell_{u_{i}} - \ell_{v_{i}}) 
        \]
        with $ v_{k} $ the successor of $u_{k} $ in $ [n] \setminus  H$. We note that these successors exist because
        $ n-m \notin H $, and $ u_{i} < n-m $.

        We begin by showing that $ (\mathcal{F}, \xi) \in W_{j,i}(\U_{r,n}) $, and here it suffices to show that
        $ \xi $ is admissible. The only non-trivial condition to check is that, if $H$ covers the flat $F$,
        then $ u_{1} < H \setminus F $. From Lemma \ref{lem:decreasingWeaklyRetral},
        $ H \setminus F = \min (H \setminus G) $. As $ u_{1} = \min_{k+1} A $, and $ \min_{1}^{k}A \subseteq G $,
        we find that $ u_{1} < H \setminus F $ as desired.

        We conclude by verifying that $ \phi(\mathcal{F}, \xi) = (k,m,A,B,C) $. 
        For notation, we will write
        $ \phi(\mathcal{F}, \xi) = (k',m',A',B',C') $.
        \begin{enumerate}
            \item As $ \min_{1}^{k} A \subseteq G $ and $ \min_{k+1} A = u_{1} \notin G $, we see that 
            $k$ is the largest integer such that $ \min_{1}^{k} (H \cup \{u_{1}, \dots, u_{i}\}) \subseteq G $,
            which is the definition of $k'$.
            \item As $ A \subseteq [n-m-1] $, we find that $ n-m \notin H $, so $ m' \leq m $. On the other hand,
            $ \left\{n-m+1, \dots,n \right\} \subseteq H $, so $m' \geq m $. Thus $ m' = m $.
            \item By definition, $ A' = (H \cup \left\{u_{1}, \dots, u_{i}\right\}) \setminus
                \left\{ n-m+1, \dots, n \right\} $. As \[ H = (A \cup \left\{n-m+1, \dots,n \right\}) \setminus
                \left\{u_{1}, \dots, u_{i} \right\},\]  we see that $ A' = A $.
            \item By definition, $ B' = \left\{u_{2}, \dots, u_{i} \right\} $. On the other hand,
            $ B = \left\{u_{1}, \dots, u_{i} \right\} \setminus \min_{k+1}A $. As $ \min_{k+1} A = u_{1} $,
            $ B' = B $.
            \item By definition $ C' = G \setminus \min_{1}^{k}G $. As $ G = C \cup \min_{1}^{k} A $ and 
            $ \min_{1}^{k} G = \min_{1}^{k} A $, we see that $ C' = C $. \qedhere
        \end{enumerate}
\end{proof}

\subsection{Counts for the augmented retral basis}\label{sec:counts-augmented-basis}
We now consider the augmented refined Hodge--Poincar\'{e} polynomials \(\H^i_{\U_{r,n}}(x)\) and the modified augmented refined Hodge--Poincar\'{e} polynomials \(\widetilde{\H}^i_{\U_{r,n}}(x)\).
We observe that \(\H^0_{\U_{r,n}}(x)\) coincides with \(\widetilde{\H}^0_{\U_{r,n}}(x)\), and both are equal to the augmented Chow polynomial \(\H_{\U_{r,n}}(x)\). We can therefore focus on the case \(i\geq 1\). 
To determine the coefficients of these polynomials we count the number of elements in the augmented retral basis from Theorems \ref{thm:main-augmented-basis} and \ref{thm:modAugBasis}. Since the basis for the modified augmented cohomology ring is a subset of the unmodified one, for \(i\geq 1\) we can write
\begin{equation}\label{eq:augmented-hp}
\H^i_{\U_{r,n}}(x) = \widetilde{\H}^i_{\U_{r,n}}(x) + \mathsf{B}^i_{\U_{r,n}}(x),
\end{equation}
where \(\mathsf{B}^i_{\U_{r,n}}(x)\) enumerates the elements of \(\mathcal B^{\IN}_{\U_{r,n}}\).

\begin{example}
We continue with Example \ref{ex:U24-fan}. By inspecting the elements of Table \ref{tab:aug-basis-U24}, we can write
\begin{align*}
    \H^0_{\U_{2,4}}(x) &= \widetilde{\H}^0_{\U_{2,4}}(x) = 1 + 5x + x^2,&  &\\
    \H^1_{\U_{2,4}}(x) &= 8x + 9x^2,  & \widetilde{\H}^1_{\U_{2,4}}(x) &= 8x + 5x^2, \\
    \H^2_{\U_{2,4}}(x) &= 4x + 6x^2,  & \widetilde{\H}^2_{\U_{2,4}}(x) &= 4x + 3x^2.
\end{align*}
\end{example}

\begin{proposition}\label{prop:formula-B-IN}
    For \(i \geq 1\),
    \[
    \mathsf{B}^i_{\U_{r,n}}(x) = \binom{n}{r+i}\binom{r+i-1}{i-1}x^r.
    \]
\end{proposition} 
\begin{proof}
    We enumerate the elements of the basis by first choosing a set \(S\) of size \(r+i\) and then partitioning it as \(S = B\sqcup J\) with \(|B| = r\) and \(|J| = i\). By the definition of \(\mathcal B^{\IN}_{\U_{r,n}}\), we build \(J\) by including the minimal element of \(S\), and then completing it with \(i-1\) of the \(r+i-1\) elements in \(S\setminus \{\min S\}\).
\end{proof}
The next result provides a formula for the modified augmented Hodge--Poincaré polynomials. 
\begin{proposition}\label{prop:formula-Htilde}
    For \(i\geq 1\),
    \[
    \widetilde{\H}^i_{\U_{r,n}}(x) = \sum_{s=0}^{r-1} \binom{n}{s} \left[\sum_{\ell=1}^{n-r-i+1}\binom{n-s-\ell}{r-s-1}\binom{n-r-\ell}{i-1} \right] A_s(x) \cdot x^{r-s}.
    \]
\end{proposition}
\begin{proof}
    If \(i\geq 1\), we count the elements of our basis as follows. We start by fixing a rank \(0\leq k\leq r-1\), and for each of the \(\binom{n}{k}\) flats of rank \(k\) we count the elements of the non-augmented retral basis for the contraction \(\U_{r,n}/F \cong \U_{r-k,n-k}\). In total we get
    \[
    \widetilde{\H}^i_{\U_{r,n}}(x) = x\sum_{k=0}^{r-1}\binom{n}{k}\underline{\H}^i_{\U_{r-k,n-k}}(x).
    \]
    By Lemma \ref{lem:hpDerangement}, this can be expanded as
    \begin{align*}
   \widetilde{\H}^i_{\U_{r,n}}(x)
    &= x\sum_{k=0}^{r-1}\binom{n}{k}\sum_{j=0}^{r-k-1}\sum_{\ell=1}^{n-r-i+1} \binom{n-k}{j}\binom{n-k-j-\ell}{r-k-j-1}\binom{n-r-\ell}{i-1}d_j(x) x^{r-k-j-1} \\
    &= \sum_{k=0}^{r-1}\sum_{s=k}^{r-1}\sum_{\ell=1}^{n-r-i+1}\binom{n}{s}\binom{s}{k}\binom{n-s-\ell}{r-s-1}\binom{n-r-\ell}{i-1}d_{s-k}(x)x^{r-s} \\
    &= \sum_{s=0}^{r-1} \binom{n}{s} \left[\sum_{\ell=1}^{n-r-i+1}\binom{n-s-\ell}{r-s-1}\binom{n-r-\ell}{i-1} \right] \sum_{k=0}^s \binom{s}{k}d_{s-k}(x) \cdot x^{r-s} \\
    &= \sum_{s=0}^{r-1} \binom{n}{s} \left[\sum_{\ell=1}^{n-r-i+1}\binom{n-s-\ell}{r-s-1}\binom{n-r-\ell}{i-1} \right] A_s(x) \cdot x^{r-s}, \\
\end{align*}
where the first three equalities follow by rearranging and reindexing the sums and the last one comes from \eqref{eq:eulerian-derangement-identity}.
\end{proof}

\section{Zeros of refined Hodge--Poincaré polynomials}\label{sec:realRooted}
In this section we prove Theorems \ref{thm:main} and \ref{thm:main-augmented}. To do so, we consider the involution $\mathcal I_{r}$ defined on polynomials of degree at most $r$
by \[\mathcal I_r(f(x)) = x^rf(x^{-1}).\] Clearly, if \(f\) is a polynomial of degree at most \(r\), then \(f\) is real-rooted if and only if \(\mathcal I_r(f)\) is, since the nonzero roots of \(f\) are the reciprocals of the nonzero roots of \(\mathcal{I}_{r}(f)\).

Let \(f(x)=a_{0}+a_{1}x+\cdots+a_{d}x^{d}\). We say that $f(x)$ is \emph{real-rooted} if all of its complex zeros are real. 
A finite sequence of nonnegative real numbers \( (a_{0},a_{1},\dots,a_{d}) \) is \emph{log-concave} if \( a_{k}^{2}\geq a_{k-1}a_{k+1} \) for every $1\leq k\leq d-1$. 
It is \emph{unimodal} if there is an index $0\leq m\leq d$ such that 
\[ 
a_{0}\leq a_{1}\leq\cdots\leq a_{m} \geq a_{m+1}\geq\cdots\geq a_{d}. 
\] 
It has no \emph{internal zeros} if $ a_{i} = 0 $ implies that $ a_{i-1} a_{i+1} = 0 $.

We say that a polynomial is log-concave, unimodal, or has no internal zeros if its coefficient sequence has the corresponding 
property. The following implications are standard. 

\begin{proposition}\label{prop:hierarchy} 
    Let $f(x)$ be a polynomial with nonnegative coefficients and no internal zeros. Then \[ f(x)\text{ is real-rooted} \quad\Longrightarrow\quad f(x)\text{ is log-concave} \quad\Longrightarrow\quad f(x)\text{ is unimodal}. \] 
\end{proposition}

Interlacing gives a useful way to compare the zeros of two real-rooted polynomials. 
Let $f(x)$ and $g(x)$ be real-rooted polynomials with positive leading coefficients. Write the zeros of $f$ and $g$, counted with multiplicity, as $\cdots\leq\alpha_{2}\leq\alpha_{1}$ and $\cdots\leq\beta_{2}\leq\beta_{1}$, respectively.
We say that $g$ \emph{weakly interlaces} $f$ if \[ 
\cdots\leq\beta_{2}\leq\alpha_{2} \leq\beta_{1}\leq\alpha_{1}. 
\]
In this case we write $ g \prec f $.

\subsection{The deranged map and the Eulerian transformation}\label{section:derangedMap}

Brändén and Solus study the zeros of Eulerian and derangement polynomials and colored derangement polynomials by defining the \emph{deranged map} \cite[Section~3.2]{branden-solus}.
This is the unique linear operator $\mathcal D:\mathbb R[x]\longrightarrow\mathbb R[x]$ satisfying 
\[ 
\mathcal D(x^{n})=d_{n}(x) 
\] 
for every $n\geq0$.

For a fixed nonnegative integer $r$, consider the basis $\mathfrak B_{r} = \left\{ x^{k}(1+x)^{r-k}:0\leq k\leq r \right\}$ of the vector space of polynomials of degree at most $r$. 
We refer to $\mathfrak B_{r}$ as the \emph{magic basis}. 
The following result of Brändén and Solus is the main real-rootedness criterion that we will use.
\begin{theorem}[{\cite[Corollary~3.7]{branden-solus}}]\label{thm:branden-solus} 
Let $f(x) = \sum_{k=0}^{r}h_{k}x^{k}(1+x)^{r-k}$, with $h_{k}\geq0$. 
Then $A_{r}(x)\prec\mathcal D(f)\prec d_{r}(x)$.
In particular, $\mathcal D(f)$ is real-rooted.
\end{theorem} 

Similarly, one can define the \emph{Eulerian transformation} \(\mathcal A^\circ: \mathbb R[x] \to \mathbb R[x]\) as the unique linear operator satisfying
\[
\mathcal A^\circ(1) = 1 \quad \text{and} \quad\mathcal A^\circ(x^n) = xA_n(x).
\]
The following result was first conjectured by Brändén and Jochemko \cite{branden-jochemko} and later proved by Athanasiadis \cite{athanasiadis}.
\begin{theorem}[{\cite[Theorem~1.1]{athanasiadis}}]\label{thm:eulerian-real-rooted}
    Let $f(x) = \sum_{k=0}^{r}h_{k}x^{k}(1+x)^{r-k}$, with $h_{k}\geq0$. 
Then $\mathcal A^\circ((1+x)^r)\prec\mathcal A^\circ(f)\prec xA_r(x)$.
In particular, $\mathcal A^\circ(f)$ is real-rooted.
\end{theorem}

\begin{remark}
    The operators \(\mathcal D\) and \(\mathcal A^\circ\) were used by Brändén and the second author to prove the real-rootedness of \(\underline{\H}_{\U_{r,n}}(x)\) \cite[Theorem~1.1]{brandenVecchiUniform} and to give an alternative proof of the real-rootedness of \(\H_{\U_{r,n}}(x)\) \cite[Theorem~3.1]{brandenVecchiUniform}.
    In a second paper, these maps were generalized to provide real-rootedness results for the (augmented) Chow polynomials of totally nonnegative posets \cite{brandenVecchiNonnegative}. These polynomials generalize $\underline{\H}_\M(x)$ and $\H_{\M}(x)$ from geometric lattices to arbitrary finite graded bounded posets.
\end{remark}

\subsection{Zeros of Hodge--Poincaré polynomials}

We start by showing the real-rootedness of the refined Hodge--Poincaré polynomials for uniform matroids. This settles \cite[Conjecture~5.8]{binderSingularCohomologyUniform}
\begin{proof}[Proof of Theorem~\ref{thm:main}]
    Fix a refined Hodge--Poincar\'{e} polynomial $ \underline{\H}^{i}_{\U_{r,n}}(x) $. If $ i=0 $, this corresponds to the Chow polynomial which is real-rooted by \cite[Theorem~1.1]{brandenVecchiUniform}. So let us assume that $ i \geq 1 $.
    Define the polynomial 
    \[ 
        f(x) = \sum_{k=0}^{r-1} \sum_{m=0}^{k} \binom{n-m-1}{r+i-m-1} \binom{i+k - m - 1}{i-1} x^{r-1-k}(1+x)^{k}.
    \]
    By Theorem \ref{thm:branden-solus}, $ \mathcal{D}(f) $ is real-rooted. Therefore $ \mathcal{I}_{r-1} \circ 
    \mathcal{D}(f) $ is real-rooted as well. We complete the proof by showing that
    $\underline{\H}^{i}_{\U_{r,n}}(x) = (\mathcal{I}_{r-1} \circ \mathcal{D})(f) $. 
    
    Expanding $ f(x) $,
    \begin{align*}
        f(x) &= \sum_{k=0}^{r-1} \sum_{m=0}^{k} \binom{n-m-1}{r+i-m-1} \binom{i+k - m - 1}{i-1}x^{r-1-k}(1+x)^{k} \\
        &= \sum_{k=0}^{r-1}\sum_{j=0}^{k}\sum_{m=0}^{k} \binom{n-m-1}{r-m+i-1} \binom{i+k-m-1}{i-1} \binom{k}{j}x^{r-1-k+j} \\
        &= \sum_{k=0}^{r-1}\sum_{j=0}^{r-k-1}\sum_{m=0}^{r-k-1}  \binom{n-m-1}{r-m+i-1} \binom{i+r-k-1-m-1}{i-1}\binom{r-k-1}{j}x^{k+j} \\
        &= \sum_{k=0}^{r-1}\sum_{j=k}^{r-1}\sum_{m=0}^{r-k-1} \binom{n-m-1}{r-m+i-1} \binom{i+r-k-1-m-1}{i-1} \binom{r-k-1}{j-k}x^{j} \\
        &= \sum_{j=0}^{r-1}\sum_{k=0}^{j}\sum_{m=0}^{r-k-1} \binom{n-m-1}{r-m+i-1} \binom{i+r-k-1-m-1}{i-1} \binom{r-k-1}{j-k}x^{j}. 
    \end{align*}
    We now observe that for every $0\leq j \leq r-1$ and $j\neq 1$ the coefficient of $x^j$ in $f(x)$ coincides with $\left|W_{j,i}(\U_{r,n})\right|$ by Lemma \ref{lem:count-retral}.
    Applying $ \mathcal{D} $ and $ \mathcal{I}_{r-1} $ yields
    \begin{align*}
        \mathcal{D}(f) &=   \sum_{j=0}^{r-1} \left| W_{j,i}(\U_{r,n})\right|d_{j}(x) \\ 
        (\mathcal{I}_{r-1} \circ \mathcal{D})(f) &= \sum_{j=0}^{r-1} \left| W_{j,i}(\U_{r,n})\right|d_{j}(x) \cdot x^{r-j-1},
    \end{align*}
    where the first equality holds because $d_1(x) = 0$ and in the second we use the fact that $d_j(x)$ is palindromic with center of symmetry \(j/2\), i.e.,
    \(\mathcal I_j(d_j(x)) = d_j(x)\)
    .
    Therefore $ \underline{\H}^{i}_{\U_{r,n}}(x) = (\mathcal{I}_{r-1} \circ \mathcal{D})(f) $ by Lemma \ref{lem:hpDerangement}.
\end{proof}

We now move on to consider the augmented case. Using \eqref{eq:augmented-hp} and Proposition \ref{prop:formula-Htilde}, it is not hard to see that the polynomials \(\H^i_{\U_{r,n}}(x)\) are not real-rooted in general.
\begin{example}
Direct computation gives
\[
\H^4_{\U_{3,8}}(x) = 140x + 372x^2 + 259x^3,
\]
which has a pair of complex-conjugate roots. Indeed, after factoring out an \(x\), the remaining degree-2 polynomial has discriminant
\[
372^2 - 4\cdot259\cdot 140 = - 6656 < 0.
\]
Unimodality also fails, although this happens for a much larger example. For example,
\[
\H^{24}_{\U_{4,30}}(x) = 1319500x + 9083815x^2 + 8789785x^3 + 8811999x^4.
\]
\end{example}
The failure of these nice distributional properties can be traced back to the failure of the quasi-projective Strong Lefschetz property, or, more specifically, to the contribution of basis elements from \(\mathcal B^{\IN}_{\U_{r,n}}\); see also Remark~\ref{rem:failure-lefschetz}. Indeed, the degree-\(4\) coefficient of \(\H^{24}_{\U_{4,30}}(x)\) is also accounting, by Proposition~\ref{prop:formula-B-IN}, for the \(\binom{30}{28}\binom{27}{23}\) elements of degree \(28\) in \(\mathcal B^{\IN}_{\U_{4,30}}\). Removing those yields a polynomial that is again unimodal and, indeed, also real-rooted. This is another reason why we want to consider the modified augmented singular cohomology ring. As a consequence of Theorem~\ref{thm:lefschetzModAug}, all polynomials \(\widetilde{\H}^i_{\U_{r,n}}(x)\) are unimodal. We will now prove Theorem~\ref{thm:main-augmented}, which states the stronger property of real-rootedness. We prepare the proof with some preliminary results.

\begin{lemma}\label{lem:Eulerian-magic-positivity}
For $1\leq i\leq n-r$, we have
\[
\begin{aligned}
&\sum_{s=0}^{r-1}\binom{n}{s}
\left[\sum_{\ell=1}^{n-r-i+1}\binom{n-s-\ell}{r-s-1}\binom{n-r-\ell}{i-1}\right]x^s \\
&\qquad= \sum_{p=0}^{r-1} \sum_{\ell=1}^{n-r-i+1} \binom{n-r-\ell}{i-1} \binom{n-\ell-p}{r-p-1} \binom{\ell-1+p}{p} x^p(1+x)^{r-1-p}.
\end{aligned}
\]
In particular, the polynomial on the left has a nonnegative expansion in the magic basis.
\end{lemma}
\begin{proof}
Interchanging the sums on the left-hand side gives
\[
\sum_{\ell=1}^{n-r-i+1} \binom{n-r-\ell}{i-1} \sum_{s=0}^{r-1} \binom{n}{s} \binom{n-s-\ell}{r-s-1}x^s.
\]
It therefore suffices to prove, for every fixed $\ell$, that
\[
\sum_{s=0}^{r-1} \binom{n}{s} \binom{n-s-\ell}{r-s-1}x^s = \sum_{p=0}^{r-1} \binom{n-\ell-p}{r-p-1} \binom{\ell-1+p}{p} x^p(1+x)^{r-1-p}.
\]
Expanding $(1+x)^{r-1-p}$ and collecting the coefficient of $x^s$ on the right-hand side yields
\begin{align*}
&\sum_{p=0}^{s} \binom{r-1-p}{s-p} \binom{n-\ell-p}{r-1-p} \binom{\ell-1+p}{p} \\
&\qquad= \binom{n-\ell-s}{r-1-s} \sum_{p=0}^{s} \binom{n-\ell-p}{s-p} \binom{\ell-1+p}{p} \\
&\qquad= \binom{n-\ell-s}{r-1-s} \binom{n}{s},
\end{align*}
where the first identity is a standard manipulation and the second is a Vandermonde convolution. Thus the two polynomials agree coefficientwise, hence the result.
\end{proof}

\begin{proof}[Proof of Theorem~\ref{thm:main-augmented}]
For \(i=0\), \(\widetilde{\H}^0_{\U_{r,n}}(x)\) is the augmented Chow polynomial, which is real-rooted by \cite[Theorem~1.10]{FMSV}. We can then focus on the case when \(i\geq 1\). The operator \(\mathcal I_r\) preserves real-rootedness, thus we study \(\mathcal I_r(\widetilde{\H}^i_{\U_{r,n}}(x))\) and prove that this polynomial has only real zeros.
Since \(A_0(x) = 1\) and \(A_s(x)\) is palindromic with center of symmetry \(\tfrac{s-1}{2}\) for \(s\geq 1\), we can now consider \(\mathcal I_r(\widetilde{\H}^i_{\U_{r,n}}(x))\) and write
\begin{align*}
&\mathcal I_r(\widetilde{\H}^i_{\U_{r,n}}(x))\\ =& \sum_{\ell=1}^{n-r-i+1}\binom{n-\ell}{r-1}\binom{n-r-\ell}{i-1} + \sum_{s=1}^{r-1} \binom{n}{s} \left[\sum_{\ell=1}^{n-r-i+1}\binom{n-s-\ell}{r-s-1}\binom{n-r-\ell}{i-1} \right] xA_s(x).
\end{align*}
This means that if we let
\[
f(x) = \sum_{s=0}^{r-1} \binom{n}{s} \left[\sum_{\ell=1}^{n-r-i+1}\binom{n-s-\ell}{r-s-1}\binom{n-r-\ell}{i-1} \right] x^s,
\]
then \(\mathcal A^\circ(f(x)) = \mathcal I_r(\widetilde{\H}^i_{\U_{r,n}}(x))\).
By Lemma \ref{lem:Eulerian-magic-positivity}, \(f(x)\) has non-negative expansion in the magic basis and by Theorem \ref{thm:eulerian-real-rooted}, the polynomial \(\mathcal A^\circ(f(x))\) is real-rooted. 
\end{proof}

\bibliographystyle{amsalpha}
\bibliography{bibliography.bib}

\end{document}